%% file: MAIN.tex
\documentclass[a4paper,UKenglish]{article} 
\usepackage{import}
\usepackage{a4wide}
\usepackage{enumerate}
\usepackage{bbm}
\usepackage[all]{xy}
\usepackage{wrapfig}

\usepackage{amsmath,amsfonts,amssymb,amsthm}
\usepackage{graphicx}
\usepackage{hyperref}
\hypersetup{colorlinks=true,linkcolor=black,citecolor=black}

\newtheorem*{theorem*}{Theorem}
\newtheorem{theorem}{Theorem}
\newtheorem{proposition}[theorem]{Proposition}
\newtheorem{lemma}[theorem]{Lemma}
\newtheorem{claim}[theorem]{Claim}
\newtheorem{corollary}[theorem]{Corollary}
\newtheorem{conjecture}[theorem]{Conjecture}

\theoremstyle{definition}
\newtheorem{definition}[theorem]{Definition}

\newtheorem{observation}[theorem]{Observation}
\newtheorem{remark}[theorem]{Remark}
\newtheorem{remarks}[theorem]{Remarks}
\newtheorem*{strategy*}{Strategy}

\newcommand{\id}{\ensuremath{\text{id}}}
\newcommand{\R}{\ensuremath{\mathbb{R}}}

\newcommand{\embeds}{\ensuremath{\hookrightarrow}}

\newcommand{\Z}{\ensuremath{\mathbb{Z}}}

\newcommand{\boundary}{\ensuremath{\partial}}

\newcommand{\interior}[1]{\ensuremath{\mathring{#1}}}

\renewcommand{\epsilon}{\ensuremath{\varepsilon}}
\newcommand{\Matrix}[1]{\left(\vcenter{\xymatrix@=0pt{ #1} }\right)}

\newcommand{\ie}{\textit{i.e.},\ }
\newcommand{\Ie}{\textit{I.e.},\ }

\renewcommand{\phi}{\varphi}

\newcommand{\homeq}{\ensuremath{\simeq}}

\newcommand{\sym}{\ensuremath{\mathfrak{S}}}
\renewcommand{\t}[1]{\ensuremath{\widetilde{#1}}}
\newcommand{\incl}{\ensuremath{\text{incl}}}
\newcommand{\iso}{\ensuremath{\cong}}

\newcommand{\define}[1]{\textbf{#1}}
\DeclareMathOperator{\conn}{conn}

\newcommand{\delprod}[2]{\ensuremath{{#1}^{#2}_{\Delta}}}
\newcommand{\delprodthin}[2]{\ensuremath{{#1}^{*#2}_{\delta}}}

\newif\ifcmts
\cmtstrue
\newcounter{sideremark}

\begin{document}

\begin{center}

\renewcommand{\thefootnote}{\fnsymbol{footnote}}

{\Large Eliminating Higher-Multiplicity Intersections, II.}
\vskip 0.1in
{\Large 
The Deleted Product Criterion in the $r$-Metastable Range}\footnotemark[1]

\vskip 0.15in

\begin{minipage}{0.25\textwidth}
\centering
{\large \sc Isaac Mabillard}
\\
\texttt{\href{mailto:imabillard@ist.ac.at}{imabillard@ist.ac.at}}
\end{minipage}
\hskip 0.1in
and
\hskip 0.1in
\begin{minipage}{0.2\textwidth} 
\centering
{\large \sc Uli Wagner}
\\ \texttt{\href{mailto:uli@ist.ac.at}{uli@ist.ac.at}}
\end{minipage}
\begin{minipage}{0.05\textwidth} 
\mbox{ }
\end{minipage}
\footnotetext[1]{Research supported by the Swiss National Science Foundation (Project SNSF-PP00P2-138948). An extended abstract of this work was presented at the 32nd International 
Symposium on Computational Geometry \cite{MabillardWagner:Elim_II_SoCG-2016}.
We would like to thank Arkadiy Skopenkov as well as the anonymous referees of the extended abstract for many helpful comments.}

\vskip 0.15in

{\it IST Austria, Am Campus 1, 3400 Klosterneuburg, Austria \hskip 0.4in}

\vskip 0.1in

{\today}

\end{center}

\begin{abstract}
Motivated by Tverberg-type problems in topological combinatorics and by classical results about embeddings (maps without double points), 
we study the question whether a finite simplicial complex $K$ can be mapped into $\R^d$ without higher-multiplicity intersections. We focus on 
conditions for the existence of \emph{almost $r$-embeddings}, i.e., maps $f\colon K \to \R^d$ such that $f(\sigma_1)\cap \dots \cap f(\sigma_r) =\emptyset$ 
whenever $\sigma_1,\dots,\sigma_r$ are pairwise disjoint simplices of $K$.

Generalizing the classical Haefliger-Weber embeddability criterion, we show that a well-known necessary \emph{deleted product condition} 
for the existence of almost $r$-embeddings is sufficient in a suitable \emph{$r$-metastable range} of dimensions: If $r d \geq (r+1) \dim K +3,$ then there exists an almost $r$-embedding $K\to \R^d$ if and only if there exists an equivariant map $\delprod{K}{r} \to_{\sym_r} S^{d(r-1)-1}$, where $\delprod{K}{r}$ is the \emph{deleted $r$-fold product} of $K$ (the subcomplex of the $r$-fold cartesian product whose cells are products of pairwise disjoint simplices of $K$) and $\sym_r$ denotes the symmetric group. This significantly extends one of the main results of our previous paper (which treated the special case where $d=rk$ and $\dim K=(r-1)k$ for some $k\geq 3$), and settles one of the main open questions raised there.

As a corollary, our result together with recent work of Filakovsk\'{y} and Vok\v{r}\'{i}nek on the homotopy classification of equivariant maps under non-free actions imply that almost $r$-embeddability of simplicial complexes is algorithmically decidable in the $r$-metastable range (in polynomial time if $r$ and $d$ are fixed).  
\end{abstract}

\vfill
\tableofcontents
\vfill
\mbox{}\clearpage

\input{introduction}
\input{two-lemmas}
\input{proof-of-the-theorem}
\input{proof-of-the-second-lemma}

\clearpage 
\appendix
\input{BB}

\bibliographystyle{alpha}
\bibliography{EliminatingMultiplePoints}

\end{document}

%% file: introduction.tex
\section{Introduction}

Let $K$ be a finite simplicial complex, and let $f\colon K\to \R^d$ be a continuous map.\footnote{For simplicity, throughout most of the paper we use the same notation for a simplicial complex $K$ and its \emph{underlying polyhedron}, relying on context to distinguish between the two when necessary.} Given an integer $r\geq 2$, we say that $y\in \R^d$ is an \define{$r$-fold point} or 
\define{$r$-intersection point} of $f$ if it has $r$ pairwise distinct preimages, i.e., if there exist $y_1,\ldots,y_r\in K$ such that $f(y_1)=\ldots=f(y_r)=y$ and $y_i\neq y_j$ for $1\leq i<j\leq r$. 
We will pay particular attention to $r$-fold points that are \define{global}\footnote{In our previous paper \cite{MabillardWagner:TverbergWhitney-2014}, we used the terminology ``\emph{$r$-Tverberg point}'' instead of ``\emph{global $r$-fold point}.''} in the sense that their preimages lie in $r$ \emph{pairwise disjoint simplices} of $K$, i.e., 
$y\in f(\sigma_1)\cap \ldots \cap f(\sigma_r)$, where $\sigma_i\cap \sigma_j=\emptyset$ for $1\leq i<j\leq r$. 

We say that a map $f\colon K\to \R^d$ is an \define{$r$-embedding} if it has no $r$-fold points, and we say that $f$ is an \define{almost $r$-embedding} if it has no \emph{global} $r$-fold points.\footnote{We emphasize that the definitions of global $r$-fold points and of almost $r$-embeddings depend on the actual simplicial complex $K$ (the specific triangulation), not just the underlying polyhedron.}

The most fundamental case $r=2$ is that of \define{embeddings} (=$2$-embeddings), i.e., injective continuous maps $f\colon K\to \R^d$.
Finding conditions for a simplicial complex $K$ to be embeddable into $\R^d$ --- a higher-dimensional generalization of graph planarity --- is a classical problem in topology (see   \cite{RepovsSkopenkov:NewResultsEmbeddingsManifoldsPolyhedra-1999,Skopenkov:EmbeddingKnottingManifoldsEuclideanSpaces-2008} for surveys) and has recently also become the subject of systematic study from a viewpoint of algorithms and computational complexity (see  \cite{MatousekTancerWagner:HardnessEmbeddings-2011,MatousekSedgwickTancerWagner:EmbeddabilityS3Decidable-2013,Cadek:Algorithmic-solvability-of-the-lifting-extension-problem-2013}).

Here, we are interested in necessary and sufficient conditions for the existence of almost $r$-embeddings.
One motivation are \emph{Tverberg-type problems} in topological combinatorics (see the corresponding subsection below). 
Another motivation is that, in the classical case $r=2$, embeddability is often proved in two steps: in the first step, the 
existence of an \define{almost embedding} (=almost $2$-embedding) is established; in the second step this almost embedding 
is transformed into an embedding, by removing \emph{local} self-intersections. Similarly, we expect the existence of 
an almost $r$-embedding to be not only an obvious necessary condition but a useful stepping stone towards the existence of 
$r$-embeddings and, in a further step, towards the existence of embeddings outside the so-called \emph{metastable range} (see Remark~\ref{remarks-main}~(c)).

\subsection{The Deleted Product Criterion for Almost 
\texorpdfstring{$r$}{r}-Embeddings}
There is a well-known \emph{necessary condition} for the existence of almost $r$-embeddings. Given a simplicial complex $K$ and $r\geq 2$, the (combinatorial) \define{deleted $r$-fold product}\footnote{
For more background on deleted products and the broader \emph{configuration space/test map} framework, see, e.g., \cite{Matousek:BorsukUlam-2003} or \cite{Zivaljevic:UserGuideEquivariantTopologyCombinatorics-96,Zivaljevic:UserGuideEquivariantTopologyCombinatorics2-98}.} 
of $K$ is defined as
$$
\delprod{K}{r} := \{(x_1,\ldots,x_r)\in \sigma_1\times \dots \times \sigma_r \mid \sigma_i \textrm{ a simplex of } K, \sigma_i\cap \sigma_j =\emptyset \textrm{ for } 1\leq i < j \leq r\}.
$$
The deleted product is a regular polytopal cell complex (a subcomplex of the $r$-fold cartesian product), whose cells are products of $r$-tuples of pairwise disjoint simplices of $K$.
\begin{lemma}[\textbf{Necessity of the Deleted Product Criterion}]
\label{lem:delprod-necessary} 
\label{lem:delprod-necessary-meta}
Let $K$ be a finite simplicial complex, and let $d ,r \geq 2$ be integers. If there exists an almost $r$-embedding $f:K \rightarrow \R^d$ 
then there exists an equivariant map\footnote{Here and in what follows, if $X$ and $Y$ are spaces on which a finite group $G$ acts (all group actions will be from the right) then we will use the notation $F\colon X\to_G Y$ for maps that are \define{equivariant}, i.e., that satisfy $F(x \cdot g)= F(x)\cdot g$ for all $x\in X$ and $g\in G$).} 
\begin{equation*}
\widetilde{f} \colon \delprod Kr \rightarrow_{\sym_r} S^{d(r-1)-1},
\end{equation*}
where $S^{d(r-1)-1}=\big\{(y_1,\ldots,y_r)\in (\R^d)^r\mid \textstyle \sum_{i=1}^r y_i=0, \sum_{i=1}^r \|y_i\|_2^2=1\big\}$,
and the symmetric group $\sym_r$ acts on both spaces by permuting components.
\end{lemma}
\begin{proof}
Given $f \colon K \rightarrow \R^d$, define $f^r\colon \delprod Kr \to (\R^d)^{r}$ by $f^r(x_1, \ldots, x_r):= (f(x_1), \ldots f(x_r))$. Then $f$ is an almost $r$-embedding iff the image of $f^r$ avoids the \emph{thin diagonal}
$\delta_{r}(\R^d):=\{ (y, \ldots, y) \mid  y \in \R^d \} \subset (\R^d)^r$.
Moreover, $S^{d(r-1)-1}$ is the unit sphere in the orthogonal complement  $\delta_{r}(\R^d)^\bot \cong \R^{d(r-1)}$, and
there is a straightforward homotopy equivalence
$\rho\colon (\R^d)^r \setminus \delta_{r}(\R^d) \simeq S^{d(r-1)-1}.$ Both $f^r$ and $\rho$ are equivariant, hence so is their composition
$\widetilde{f}:=\rho\circ f^r\colon \delprod{K}{r}\to_{\sym_r} S^{d(r-1)-1}.$
\end{proof}
Our main result is that the converse of Lemma~\ref{lem:delprod-necessary} holds in a wide range of dimensions.
\begin{theorem}[\textbf{Sufficiency of the Deleted Product Criterion in the $r$-Metastable Range}] 
\label{thm_Hae_Web_an}
\label{thm_Hae_Web_an_restated} 
Let $m,d,r\geq 2$ be integers satisfying 
\begin{equation}
\label{eq:r-metastable}
\hfill r d \ge (r+1) m + 3.\hfill 
\end{equation}
Suppose that $K$ is a finite $m$-dimensional simplicial complex and that there exists an equivariant map $F:\delprod Kr \rightarrow_{\sym_r} S^{d(r-1)-1}$. Then
there exists an almost $r$-embedding $f\colon K\to \R^d$.  
\end{theorem}

\begin{remarks}
\label{remarks-main}
\begin{enumerate}[(a)]
\item When studying almost $r$-embeddings, it suffices to consider maps $f\colon K\to \R^d$ that are \emph{piecewise-linear}\footnote{Recall that $f$ is PL if there is some subdivision $K'$ of $K$ such that $f|_\sigma$ is affine for each simplex $\sigma$ of $K'$.} (\emph{PL}) and \emph{in general position}.\footnote{Every continuous map 
$g\colon K\to \R^d$ can be approximated arbitrarily closely by PL maps in general position, and if $g$ is an almost $r$-embedding, then the same holds for any map sufficiently close to $g$.} 
\item Theorem~\ref{thm_Hae_Web_an} is trivial for \emph{codimension} $d-m\leq 2$. Indeed, if $r,d,m$ satisfy \eqref{eq:r-metastable} and, additionally, $d-m\leq 2$ then a straightforward calculation shows that $(r-1)d>rm$, so that a map $K\to \R^d$ in general position has no $r$-fold points.
\item The special case $r=2$ of Theorem~\ref{thm_Hae_Web_an} corresponds to the classical \emph{Haefliger--Weber Theorem} \cite{Haefliger:Plongements-differentiables-dans-le-domaine-stable-1962,Weber:Plongements-de-polyhedres-dans-le-domaine-metastable-1967}, which guarantees that for $2d\geq 3m+3$ the existence of an equivariant map $\delprod K 2 \rightarrow_{\sym_2} S^{d-1}$ guarantees the existence of an almost embedding $f\colon K\to \R^d$. The condition $2d\geq 3m+3$ is often referred to as the \define{metastable range}; correspondingly, we call Condition~\eqref{eq:r-metastable} the \define{$r$-metastable range}.\footnote{We remark that the terminology ``\emph{$k$-metastable range}''
also appears in a different setting, namely for links of a finite number of sheres $S^{p_1},S^{p_2},\ldots,S^{p_\ell}$ smoothly embedded in $S^m$, see \cite[Appendix~9.1]{haefliger1966enlacements}. The definition of the 
$r$-metastable range in the present paper and the $k$-metastable range in \cite{haefliger1966enlacements} are \emph{not} the same, but since the contexts are different, we hope that no confusion will arise.}

In the metastable range, an almost embedding can further be turned into an embedding by a delicate geometric construction~\cite{Skopenkov_deleted_product,Weber:Plongements-de-polyhedres-dans-le-domaine-metastable-1967}, which fails outside of the metastable range \cite{Segal:Quasi-embeddings-and-embeddings-of-polyhedra-in-mathbb-Rsp-m-1992}.

Turning almost $r$-embeddings into $r$-embeddings seems to be an even more subtle problem, which
we plan to pursue in a future paper.

\item Theorem~\ref{thm_Hae_Web_an} significantly extends one of the main results of our previous paper \cite[Thm.~7]{MabillardWagner15} (see also the extended abstract \cite[Thm.~3]{MabillardWagner:TverbergWhitney-2014}), which 
treated the special case $(r-1)d=rm$, $d-m\geq 3$, and settles one of the open questions raised there.
\end{enumerate}
\end{remarks}

\subsection{Algorithmic Consequences}
Very recently, Filakovsk\'{y} and Vok\v{r}\'{\i}nek~\cite{FilakovskyVokrinek} obtained the following algorithmic result regarding equivariant maps between $G$-spaces.\footnote{Their result builds on and significantly generalizes earlier results in the non-equivariant setting \cite{Cadek:Computing-all-maps-into-a-sphere-2014,Cadek:Polynomial-time-computation-of-homotopy-groups-and-Postnikov-systems-2014} and for \emph{free} actions \cite{Cadek:Algorithmic-solvability-of-the-lifting-extension-problem-2013}, respectively.}
If a group $G$ acts on a space and $H\leq G$ is a subgroup, we denote by $X^H=\{x\in X\colon xh=x , \, \forall h\in H\}$ the set of $H$-fixed points of $X$. Furthermore, we write $\conn(Y)$ for the \emph{connectivity} of a space $Y$ (i.e., the maximum $k$ such that every map $S^i \to Y$ can be extended to a map $B^{i+1}\to Y$).
\begin{theorem}[{\cite{FilakovskyVokrinek}}]
\label{thm:FV}
There is an algorithm with the following specifications: 

Let $X$ and $Y$ be finite simplicial complexes, and let $G$ be a finite group that acts (simplicially) on $X$ and $Y$. 
Suppose that $\dim (X^H) \leq 2 \conn(Y^H)+1$ for every subgroup $H\leq G$.\footnote{Here, we use the convention 
that $\dim \emptyset = -\infty = \conn(\emptyset)$. In other words the condition is trivially satisfied if $X^H=\emptyset$, 
and if $Y^H=\emptyset$, then we must have $X^H =\emptyset$ as well.}

Given $X$ and $Y$ (and the simplicial maps specifying the group action) as input, the algorithm decides whether there exists
a $G$-equivariant map $X\to_G Y$.

Moreover, if $G$ and the dimension of $X$ are fixed, then the algorithm runs in polynomial time.
\end{theorem}
As an immediate consequence of Lemma~\ref{lem:delprod-necessary} and Theorems~\ref{thm_Hae_Web_an} and \ref{thm:FV}, one obtains:
\begin{corollary} 
In the $r$-metastable range \eqref{eq:r-metastable}, there is an algorithm that, 
given as input a finite $m$-dimensional simplicial complex $K$ and parameters $r$ and $d$,
decides whether there is an almost $r$-embedding of $K$ to $\R^d$. If $r$ and $d$ are fixed, the algorithm runs in polynomial time.
\end{corollary}
\begin{proof} By Lemma~\ref{lem:delprod-necessary} and Theorem~\ref{thm_Hae_Web_an}, there exists an almost $r$-embedding of 
$K$ to $\R^d$ iff there exists an $\sym_{r}$-equivariant map from $X:=\delprod{K}{r}$ to $Y:=S^{d(r-1)-1}$. Thus, we only need to check that the assumptions of Theorem~\ref{thm:FV} are satisfied in this setting. Both the deleted product $X$ and the sphere $Y$ can be triangulated in such a way that $G=\sym_{r}$ acts by simplicial maps. Moreover, the action of $\sym_{r}$ on $X$ is free, hence $X^H=\emptyset$ if $H\leq  \sym_{r}$ is not the trivial group $\{e\}$ consisting only of the identity, and $X^{\{e\}}=X$. Thus, the condition $\dim X^H \leq 2\conn (Y^H)+1$ is trivially satisfied whenever $H\neq \{e\}$, and for $H=\{e\}$ it is satisfied since $\dim X \leq rm \leq 2\conn(Y)+1=2d(r-1)-1$ in the $r$-metastable range.
\end{proof}

\subsection{Background and Motivation: Tverberg-Type Problems and Beyond}

Methods from algebraic topology and the general framework of \emph{configuration spaces} and \emph{test maps} have been very successfully used in discrete mathematics and theoretical computer science, see, e.g., 
\cite{Zivaljevic:UserGuideEquivariantTopologyCombinatorics-96,Zivaljevic:UserGuideEquivariantTopologyCombinatorics2-98,Matousek:BorsukUlam-2003,Blagojevic:Using-equivariant-obstruction-theory-in-combinatorial-geometry-2007,Karasev:Topological-methods-in-combinatorial-geometry-2008} for surveys. In particular, \emph{equivariant obstruction theory} and, more generally \emph{equivariant homotopy theory}, provide powerful tools for deciding whether suitable test maps exist. However, in cases where the existence of a test map does not settle the problem, further geometric ideas are needed. The general philosophy and underlying idea here and in the two companion papers \cite{MabillardWagner15,Avvakumov:Eliminating-Higher-Multiplicity-Intersections-III.-2015} is to complement equivariant methods by methods from \emph{geometric topology}, in particular \emph{piecewise-linear topology}.
\medskip

The initial motivation and first application for the predecessor paper \cite{MabillardWagner15} was the long-standing 
\begin{conjecture}[Topological Tverberg Conjecture]\label{conj:Tvb}
Let $r\geq 2$, $d\geq 1$, and $N=(d+1)(r-1)$. Then there is no almost $r$-embedding of the $N$-simplex $\sigma^N$ to $\R^d$.
\end{conjecture}
This conjecture, proposed by Bajmoczy and B\'ar\'any~\cite{Bajmoczy:On-a-common-generalization-of-Borsuks-and-Radons-1979} and Tverberg~\cite[Problem~84]{Gruber:Problems-in-geometric-convexity-1979} as a topological generalization of a classical theorem of Tverberg in convex geometry~\cite{Tverberg:A-generalization-of-Radons-theorem-1966}, had been proved by Bajmoczy and B\'ar\'any~\cite{Bajmoczy:On-a-common-generalization-of-Borsuks-and-Radons-1979} for $r=2$, by B\'ar\'any, Shlosman, and Sz\H{u}cs~\cite{Barany:On-a-topological-generalization-of-a-theorem-of-Tverberg-1981} for all primes $r$, and by \"Ozaydin~\cite{Ozaydin:Equivariant-maps-for-the-symmetric-group-1987} for prime powers $r$, but the case of arbitrary $r$ had remained 
elusive. An important reason was that \"Ozaydin~\cite[Thm.~4.2]{Ozaydin:Equivariant-maps-for-the-symmetric-group-1987} had shown
that for $r$ is \emph{not} a prime power there does exist an equivariant map $F\colon \delprod {(\sigma^N)}{r} \to_{\sym_{r}} S^{d(r-1)-1}$, so that Lemma~\ref{lem:delprod-necessary} cannot be directly used to show that there is no almost $r$-embedding.

In \cite{MabillardWagner:TverbergWhitney-2014,MabillardWagner15}, we proposed a new approach to the conjecture, namely the idea of constructing counterexamples, i.e., almost $r$-embeddings $\sigma^{N}\to \R^d$, when $r$ is not a prime power, by combining \"Ozaydin's result with the sufficiency of the deleted product product in the special case $\dim K=\frac{r-1}{r}d$, $d-\dim K\geq 3$ of Theorem~\ref{thm_Hae_Web_an} (\cite[Thm~3]{MabillardWagner:TverbergWhitney-2014} and \cite[Thm~7]{MabillardWagner15}).
At the time the extended abstract  \cite{MabillardWagner:TverbergWhitney-2014} appeared, there remained what seemed to be a serious obstacle to completing this approach: the assumption of \emph{codimension $d-\dim K\geq 3$} is not satisfied for $K=\sigma^N$.
Frick \cite{Frick:Counterexamples-to-the-topological-Tverberg-conjecture-2015} 
observed that this ``codimension $3$ barrier'' can be overcome by a combinatorial trick discovered independently 
independently by Gromov~\cite[p.~445-446]{gromov2010singularities} 
and Blagojevi\'c-Frick-Ziegler
\cite{Blagojevic:Tverberg-plus-constraints-2014})
and that therefore the results of \cite{Ozaydin:Equivariant-maps-for-the-symmetric-group-1987},
\cite{gromov2010singularities,Blagojevic:Tverberg-plus-constraints-2014} and \cite{MabillardWagner:TverbergWhitney-2014,MabillardWagner15} combined yields counterexamples to the topological Tverberg conjecture for $d \ge 3r + 1$ whenever $r$ is not a prime power, see \cite{Blagojevic:Barycenters-of-Polytope-Skeleta-and-Counterexamples-2015}. In the full version \cite{MabillardWagner15}, another solution to
the codimension $3$ obstacle is given, leading to counterexamples for $d \geq 3r$. In joint work with Avvakumov and Skopenkov \cite{Avvakumov:Eliminating-Higher-Multiplicity-Intersections-III.-2015}, we recently improved this further and obtained counterexamples for $d\geq 2r$, using an extension (for $r\geq 3$) of \cite[Thm.~7]{MabillardWagner15} to \emph{codimension 2}.
\smallskip

For a more detailed history of the topological Tverberg conjecture and the construction of counterexamples, we refer to the surveys and discussions in 
\cite{Barany:Tverbergs-theorem-at-50:-extensions-and-counterexamples-2016,skopenkov2016UserGuide,Blagojevic:Beyond-the-Borsuk-Ulam-theorem:-The-topological-2016,simon2015average,jojic2016topology}.
\smallskip

There are numerous close relatives and other variants of (topological) Tverberg-type problems and results  
\cite{Barany:On-the-number-of-halving-planes-1990,Barany:A-colored-version-of-Tverbergs-theorem-1992,Zivaljevic:The-colored-Tverbergs-problem-and-complexes-of-injective-functions-1992,Zivaljevic:UserGuideEquivariantTopologyCombinatorics2-98,Blagojevic:Optimal-bounds-for-the-colored-Tverberg-problem-2009,
Sarkaria:A-generalized-van-Kampen-Flores-theorem-1991,Volovikov:On-the-van-Kampen-Flores-theorem-1996,
Blagojevic:Tverberg-plus-constraints-2014,Blagojevic:Beyond-the-Borsuk-Ulam-theorem:-The-topological-2016,Barany:Tverbergs-theorem-at-50:-extensions-and-counterexamples-2016}.
These can be seen as \emph{generalized nonembeddability results} or \emph{problems} and typically state that a particular complex $K$ (or family of complexes) does not admit an almost $r$-embedding to $\R^d$. Theorem~\ref{thm_Hae_Web_an} provides a general 
necessary and sufficient condition for such topological Tverberg-type results in the $r$-metastable range.

As an application of a different flavor, Frick~\cite{Frick:Intersection-patterns-of-finite-sets-2016} recently found a connection between almost $r$-embeddings and chromatic numbers of certain Kneser hypergraphs and used Theorem~\ref{thm_Hae_Web_an} to prove lower bounds for these chromatic numbers in certain cases.  

\subsection*{Further Questions and Future Research}
\begin{enumerate}[(a)]
\item \emph{Beyond the $r$-Metastable Range.} Is condition~\eqref{eq:r-metastable} needed in Theorem~\ref{thm_Hae_Web_an}? In the case $r=2$, it is known that for $d\geq 3$, the Haefliger--Weber Theorem fails outside the metastable range:  for every pair $(m,d)$ with $2d < 3m+3$ and $d\geq 3$, there are examples   \cite{Mardesic:varepsilon-Mappings-and-generalized-manifolds-1967,Segal:Quasi-embeddings-and-embeddings-of-polyhedra-in-mathbb-Rsp-m-1992,Freedman:van-Kampens-embedding-obstruction-is-incomplete-for-2-complexes-in-bf-R4-1994,Segal:Embeddings-of-polyhedra-in-mathbb-Rm-and-the-deleted-product-obstruction-1998,GoncalvesSkopenkov:EmbeddingsHomologyEquivalentManifolds-2006} of $m$-dimensional complexes $K$ such that $\delprod{K}{2}\to_{\sym_2} S^{d-1}$ but $K$ does not embed into $\R^d$. Moreover, in the case $r=2$, $m=2$ and $d=4$, the examples do not even admit an almost embedding into $\R^4$, see \cite{Avvakumov:Eliminating-Higher-Multiplicity-Intersections-III.-2015}. 

On the other hand, as remarked above, in \cite{Avvakumov:Eliminating-Higher-Multiplicity-Intersections-III.-2015} the following extension of \cite[Thm.~7]{MabillardWagner15} is proved: if $r\geq 3$ $d=2r$, and $m=2(r-1)$, then a finite $m$-dimensional complex $K$ admits an almost $r$-embedding if and only if there exists an equivariant map $\delprod{K}{r}\to_{\sym_r} S^{d(r-1)-1}$. 

It would be interesting to know whether there is analogous extension (for $r\geq 3$) of Theorem~\ref{thm_Hae_Web_an} that is nontrivial in codimension $d-m=2$.

\item \emph{The Planar Case and Hanani--Tutte}. In the classical setting ($r=2$) of embeddings, the case $d=2, m=1$ of \emph{graph planarity} is somewhat exceptional: the parameters lie outside the ($2$-fold) metastable range, but the existence of an equivariant map $F\colon \delprod{K}{2}\to_{\sym_2} S^1$ is sufficient for a graph $K$ to be planar, by the \emph{Hanani--Tutte Theorem}\footnote{The existence of an equivariant map implies, via standard equivariant obstruction theory, that there exists a map from the graph $K$ into $\R^2$ such that the images of any two disjoint (\emph{independent}) edges intersect an even number of times, which is the hypothesis of the Hanani--Tutte Theorem.} \cite{Hanani:UnplattbareKurven-1934,Tutte:TowardATheoryOfCrossingNumbers-1970}. The classical proofs of that theorem rely on \emph{Kuratowski's Theorem}, but recently \cite{Pelsmajer:Removing-even-crossings-2007,Pelsmajer:Removing-independently-even-crossings-2010}, more direct proofs have been found that do not use forbidden minors. It would be interesting to know whether there is an analogue of the Hanani--Tutte theorem for almost $r$-embeddings 
of $2$-dimensional complexes in $\R^2$, as an approach to constructing counterexamples to the topological Tverberg conjecture in dimension $d=2$.
We plan to investigate this in a future paper.
\end{enumerate}

\subsection*{Structure of the Paper}
The remainder of the paper is devoted to the proof of Theorem~\ref{thm_Hae_Web_an}. By Lemma~\ref{lem:delprod-necessary}, we only need to show that the existence of an equivariant map $\delprod{K}{r}\to_{\sym_r} S^{d(r-1)-1}$ implies the existence of an almost $r$-embedding $K\to \R^d$. Moreover, by Remarks~\ref{remarks-main}~(b) and (d), we may assume, in addition to the parameters being in the $r$-fold metastable range, that the \emph{codimension} $d-m$ of the image of $K$ in $\R^d$ is at least $3$, and that the intersection multiplicity $r$ is also at least $3$. Thus, we will work under the following hypothesis:
\begin{equation} 
\label{eq_metastable}
\hfill 
r d \ge (r+1) m + 3, \quad d-m \ge 3, \quad \text{and} \quad r \ge 3.
\hfill
\end{equation}

The proof of Theorem~\ref{thm_Hae_Web_an} is based on two main lemmas: Lemma~\ref{lem_reduction} (\emph{Reduction Lemma}) reduces the situation to a single $r$-tuple of pairwise disjoint simplices of $K$, and Lemma~\ref{lem_whitney_trick} (\emph{Local Disjunction Lemma}) solves that reduced situation (this is a generalization of the \emph{Weber--Whitney Trick} to multiplicity $r$). In Section~\ref{s:two-lemmas}, we give the precise (and somewhat technical) statements of these lemmas, along with some background, and prove the Reduction Lemma~\ref{lem_reduction}. In Section~\ref{s:proof-thm}, we show how to prove Theorem~\ref{s:proof-thm} using these lemmas, before proving the Local Disjunction Lemma~\ref{lem_whitney_trick} (the core of the paper)
in Section~\ref{s:proof-second-lemma}.

%% file: two-lemmas.tex
\newcommand{\footnotehidden}[1]{\footnote{#1}}

\section{Main Lemmas: Reduction \& Local Disjunction} 
\label{s:two-lemmas}
In this section, we formulate the two main lemmas on which the proof of Theorem~\ref{thm_Hae_Web_an_restated} rests.


We work in the \emph{piecewise-linear} (\emph{PL}) category (standard references are \cite{Zeeman:Seminar-on-combinatorial-topology-1966,Rourke:Introduction-to-piecewise-linear-topology-1982}).
All manifolds (possibly with boundary) are PL-manifolds (can be triangulated as locally finite simplicial complexes such that the link of every nonempty face is either a PL-sphere or a PL-ball), and all maps between \define{polyhedra} (geometric realizations of simplicial complexes) are PL-maps (\ie  simplicial on sufficiently fine subdivisions).\footnote{The PL assumption is no loss of generality: if  
$K$ is a finite simplicial complex and $f\colon K\to \R^d$ is an almost $r$-embedding then $f$ can be slightly perturbed to a PL map with the same property.
} In particular, all balls are PL-ball and all spheres are PL-spheres (PL-homeomorphic to a simplex and the boundary of a simplex, respectively).
A submanifold $P$ of a manifold $Q$ is \define{properly embedded} if $\partial P=P\cap \partial Q$.
The \define{singular set} of a PL-map $f$ defined on a polyhedron $K$ is the closure in $K$ of the set of points at which $f$ is not injective.

One basic fact that we will use for the proofs of both Lemmas~\ref{lem_reduction} and~\ref{lem_whitney_trick} is the following version of \textit{engulfing} \cite[Ch.~VII]{Zeeman:Seminar-on-combinatorial-topology-1966}:
\begin{theorem}[Engulfing, {\cite[Ch.~VII, Thm.~20]{Zeeman:Seminar-on-combinatorial-topology-1966}}]
\label{thm_Zeeman}
Let $M$ be an $m$-dimensional $k$-connected manifold with $k \le m-3$. Let $X$ a compact $x$-dimensional subpolyhedron 
in the interior of $M$. If $x \le k$, then there exists a collapsible subpolyhedron $C$ in the interior of $M$ with $X \subseteq C$ and $\dim (C ) \le x+1$.
\end{theorem}
The collapsible polyhedron $C$ can be thought of as an analogue of a ``cone'' over $X$.


\begin{lemma}[Reduction Lemma] 
\label{lem_reduction}
Let $m,d,r$ be three positive integers satisfying \eqref{eq_metastable}.
Suppose $f\colon K \rightarrow \R^d$ is a map in general position, and $\sigma_1, \ldots , \sigma_r$ be pairwise disjoint simplices of $K$  of dimension $s_1, \ldots, s_r \leq m$ such that 
\[
f|_{\sigma_i}^{-1} (f (\sigma_1) \cap \cdots \cap f (\sigma_r) )
\]
is contained in the \emph{interior} of each simplex $\sigma_i$.
Then there exists a ball $B^d$ in $\R^d$ such that 
\begin{enumerate}
\item $B^d$ intersects each $f(\sigma_i)$ in a ball that is properly embedded in $B^d$, and that avoids the image of the singular set of $f|_{\sigma_i}$, 
as well as $f(\partial \sigma_i)$;
\item $B^d$ contains $f(\sigma_1) \cap \cdots \cap f (\sigma_r)$ in its interior; and 
\item $B^d$ does not intersect any other parts of the image $f(K)$.
\end{enumerate}
\end{lemma}
\begin{proof}
Let us consider $S_i := f^{-1} (f (\sigma_1) \cap \cdots \cap f (\sigma_r)) \cap \sigma_i$. By general position \cite[Thm~5.4]{Rourke:Introduction-to-piecewise-linear-topology-1982} this is a polyhedron of dimension at most
$
s_1 + \cdots + s_r - (r-1) d \le r m - (r-1) d. 
$ 
By Theorem~\ref{thm_Zeeman}, we find $C_i \subseteq \sigma_i$ collapsible, containing $S_i$, and of dimension at most $r m - (r-1) d + 1$. Figure~\ref{fig_C_1} illustrates the case $r=3$.

\begin{figure}
\begin{center}
\includegraphics[scale=1]{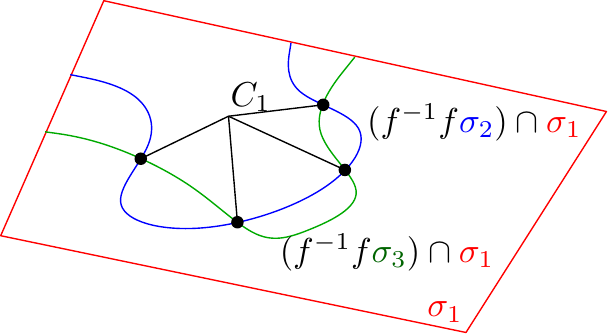}
\caption{For $r=3$, the construction of $C_1$ inside of $\sigma_1$. The collapsible polyhedron $C_1$ is a ``cone'' over the triple intersection set $S_1$ (which consists of four isolated points in the picture).}
\label{fig_C_1}
\end{center}
\end{figure}

The dimension of the singular set of $f|_{\sigma_i}$ is at most $2s_i -d$. Hence, $C_i$ is disjoint from it since
$
(r m - (r-1) d + 1) + (2s_i - d) -s_i \leq (r+1)m - r d + 1,
$ 
which is negative in the metastable range. Thus,  $f$ is injective in a neighbourhood of $C_i$.

Again by Theorem~\ref{thm_Zeeman}, we find in $\R^d$ a collapsible polyhedron $C_{\R^d}$ of dimension at most $r m - (r-1) d + 2$ and containing $f (C_1) \cup \cdots \cup f (C_r)$. Figure~\ref{fig_C_1_2_3} illustrates the construction for $r=3$.
By general position we have the following properties:
\begin{enumerate}
\item $C_{\R^d}$ intersects $ f(\sigma_i)$ exactly in $f(C_i)$. Indeed, in the metastable range,
$rm - (r-1) d + 2 + s_i - d \le (r+1) m - r d + 2 < 0.$
\item $C_{\R^d}$ does not intersect any other part of $f(K)$ (by a similar computation). 
\end{enumerate}

\begin{figure}
\begin{center}
\includegraphics[scale=1]{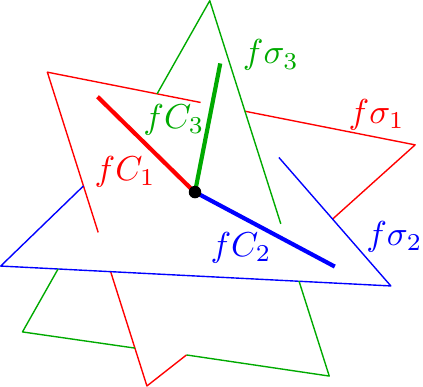}
\caption{For $r=3$, the polyhedron $C_{\R^d}$ is a ``cone'' over $fC_1 \cup f C_2 \cup f C_3$.}
\label{fig_C_1_2_3}
\end{center}
\end{figure}

We take a small \define{regular neighbourhood} \cite[Ch.~3]{Rourke:Introduction-to-piecewise-linear-topology-1982} $B$ of $C_{\R^d}$, which still avoids the singular set of each $f|_{\sigma_i}$ as well as other parts of $f(K)$. This regular neighbourhood is a ball, since $C_{\R^d}$ is collapsible. The intersection $B\cap f(\sigma_i)$ is a regular neighbourhood of $f(C_i)$ which is also a collapsible space, hence $B\cap f(\sigma_i)$ is a ball (properly contained in $B$).
\end{proof}

An \define{ambient isotopy} $H$ of a PL-manifold $X$ is a collection of homeomorphisms $H_t:  X \rightarrow X$  for $t \in [0,1]$, which vary continuously with $t$, and with $H_0 = \id$. We say that an ambient isotopy $H$ \define{throws} a subspace $Y \subseteq X$ onto $Z$ if $H_1(Y)=Z$, see \cite[Ch.~V]{Zeeman:Seminar-on-combinatorial-topology-1966}. 

We say that an ambient isotopy $H$ of $X$ is \define{proper} if  $H_t |_{\boundary X} = \id_{\boundary X}$  for all~$t$.

\begin{definition}\label{def_intersection_class}
Let $d,r \ge 2$ be integers, let $\sigma_1, \ldots , \sigma_r$ be balls of dimensions $s_1, \ldots, s_r$. We define
\[
s:=s_1+\ldots +s_r.
\]
Let $f$ be a continuous map, mapping the disjoint union of the $\sigma_i$ to a $d$-dimensional ball $B^d$, i.e., 
\[
f : \sigma_1 \sqcup \cdots \sqcup \sigma_r \to B^d.
\]

We define the \define{Gauss map $\t f$} associated to $f$
\[
\t f : \sigma_1\times \cdots \times \sigma_r \to B^d \times \cdots \times B^d, \quad \mbox{by} \quad (x_1,...,x_r)\mapsto (fx_1,...,fx_r),
\]

If, for each $i=1,...,r$,
\[
f\sigma_1 \cap \cdots \cap f \boundary \sigma_i \cap \cdots \cap  f\sigma_r = \emptyset.
\]
then  $\t f \boundary (\sigma_1 \times  \cdots \times \sigma_r) \subset B^d \times \cdots \times B^d$, avoids the \define{thin diagonal} $\delta_r (B^d) = \{ (x,\ldots , x ) \; | \; x \in B^d\}$ of $B^d$. Thus,
\begin{equation}\label{eq_gauss_map}
\boundary (\sigma_1 \times  \cdots \times \sigma_r)   \rightarrow
(B^d \times \cdots \times B^d) \setminus \delta_r (B^d).
\end{equation}

Observe that $\boundary (\sigma_1 \times  \cdots \times \sigma_r)\cong S^{s-1}$, where $s:=\sum_is_i$, and $(B^d \times \cdots \times B^d) \setminus \delta_r (B^d)$ is homotopy equivalent to $S^{d(r-1)-1}$. Therefore, the map \eqref{eq_gauss_map} defines an element 
\[
\alpha(f) \in \pi_{s-1} ( S^{d(r-1)-1} ),
\]
which we call \define{intersection class of $f$}.
\end{definition}

\begin{lemma}[Local Disjunction Lemma] \label{lem_whitney_trick} 
Let $m,d,r$ be three positive integers satisfying \eqref{eq_metastable}.

Let $\sigma_1, \ldots , \sigma_r$ be balls of dimensions $s_1, \ldots, s_r \leq m$ properly contained in a $d$-dimensional ball $B$ and with
$\sigma_1 \cap \cdots \cap \sigma_r $ in the \emph{interior} of $B$.
\begin{enumerate}
\item  \label{lem_whitney_part_I}
Let us denote by $\alpha$ the intersection class of the map $\sigma_1\sqcup\cdots\sqcup \sigma_r \to B^d$.

If $\alpha = 0 $, then there exists $(r-1)$ proper ambient isotopies of $B$ 
that we can apply to $\sigma_1, \ldots , \sigma_{r-1}$, respectively, to remove the $r$-intersection set; \ie  there exist $(r-1)$ proper isotopies $H^1_t, \ldots, H^{r-1}_t$ of $B$ throwing 
$\sigma_i$ onto $\sigma_i' := H^i_1 \sigma_i$ and such that
\[
\hfill
\sigma_1'  \cap \cdots \cap \sigma_{r-1}' \cap \sigma_r = \emptyset.
\hfill
\] 

\item \label{lem_whitney_part_II}
Let us assume that $\sigma_1 \cap \cdots \cap \sigma_r = \emptyset$ and $\sigma_2\cap \cdots \cap \sigma_r \neq \emptyset$, and let $z \in \pi_{s} ( S^{d(r-1)-1} )$.

There exists $J_t$ a proper ambient isotopy of $B$ such that 
\begin{itemize}
\item 
$
J_1 \sigma_1 \cap \sigma_2 \cap \cdots \cap \sigma_{r-1} \cap \sigma_r = \emptyset,
$
\item 
The intersection class of $f$ is $z$, where
\[
f: (\sigma_1 \times I)  \sqcup \sigma_2 \sqcup \cdots \sqcup \sigma_r 
 \to 
 B^d
\]
is defined as the inclusion on $\sigma_i$ for $i\ge 2$, and for $(x,t) \in \sigma_1 \times I$, $f(x,t) = J_t(x)$.
\end{itemize}
\end{enumerate}
\end{lemma}

\begin{remark}
\begin{itemize}
\item
The proof of Lemma~\ref{lem_whitney_trick} is the technical core of the paper and will be given in Section~\ref{s:proof-second-lemma}.
For $r=2$, Lemma~\ref{lem_whitney_trick} already appears in Section 4 of Weber's thesis \cite{Weber:Plongements-de-polyhedres-dans-le-domaine-metastable-1967}. Our contribution in the present paper is to show that the result holds for any $r\ge 3$.
\item  Roughly speaking, 
Part~\ref{lem_whitney_part_I} of Lemma~\ref{lem_whitney_trick} means that if the intersection class vanishes, then one can resolve the $r$-intersection set.

Part~\ref{lem_whitney_part_II} means that each element of $\pi_{s} ( S^{d(r-1)-1} )$ can be obtained by moving from a fixed solution to a new solution.
\item 
In our previous paper \cite{MabillardWagner15}, we consider the special case when all the global $r$-intersection points are isolated (i.e., the $r$-intersections are $0$-dimensional).  The ``elimination'' of these isolated $r$-intersections is achieved in two steps:
\begin{enumerate}[(1)]
\item \label{elim_isolated_I}
First, we obtain the algebraic cancellation of the $r$-intersection points by ``finger moves'':  we modify a given map $f :K^m\to \R^d$ such that for each $r$-tuples of pairwise disjoint cells $\sigma_1,\ldots, \sigma_r$ of $K$, the intersection $f\sigma_1\cap \cdots\cap f \sigma_r$ consists of \emph{pairs of points of opposite intersection signs} (hence, algebraically, they ``cancel'').
\item \label{elim_isolated_II}
In a second step, we geometrically cancel each pair of $r$-intersection points of opposite sign, and for this, we use an $r$-fold version of the Whitney Trick (a special case of the Local Disjunction Lemma). Hence, we obtain $f \sigma_1 \cap \cdots \cap f \sigma_r = \emptyset$.
\end{enumerate}
In other words, for the special case consider in our previous paper, the proof decomposes naturally into two steps: (1) first a ``linking step'' when we link cell together to introduce new $r$-intersection points (and therefore obtain the ``algebraic cancellation'' of the $r$-intersection points), (2) secondly, in an ``unlinking step'' we translate that algebraic cancellation into geometry (i.e., from $\mathrm{intersection} =0$, we obtain $\mathrm{intersection}= \emptyset$).

In the present paper, these two steps are not so disjoint anymore: multiple cases of global $r$-intersection points can occur, resulting in singular set of various dimension (no only isolated points). Therefore, we will have to merge the two steps (1) and (2): In our construction, we will first ``unlink'' the $r$-intersection points of a given  $r$-tuple of cells (i.e., remove their $r$-intersection points), and immediately after we will ``link'' this $r$-tuple in order to permit the unlinking of $r$-tuples of higher dimension. (See both parts of Lemma~\ref{lem_whitney_trick}: Part~\ref{lem_whitney_part_I} corresponds to the ``unlinking'', and Part~\ref{lem_whitney_part_II} corresponds to the ``linking'').

\end{itemize}
\end{remark}

%% file: proof-of-the-theorem.tex
\section{Combining Reduction and Local Disjunction} 
\label{s:proof-thm}

Here, we show how to use Lemmas~\ref{lem_reduction} and \ref{lem_whitney_trick} to prove the main theorem. The inductive argument used in the proof mirrors that of Section 5 in Weber's thesis~\cite{Weber:Plongements-de-polyhedres-dans-le-domaine-metastable-1967}, where Theorem~\ref{thm_Hae_Web_an_restated} is proven for $r=2$.

Let us recall that given a map $f \colon K^m \to \R^d$, we can induce 
\[
\t f \colon \delprod K r \to_{\sym_r} \R^{r \times d} \quad \mbox{by} \quad (x_1, \ldots,x_r) \mapsto (fx_1, \ldots, fx_r),
\]
whose image avoids the diagonal $\{(x,\ldots,x) \; | \; x \in \R^d\}$ if and only if $f$ is an $r$-almost embedding.

We condense the two parts of Lemma~\ref{lem_whitney_trick} into the following technical statement:
\begin{lemma}[Inductive step]\label{lem_inductive_step}
Let $f \colon K^m \to \R^d$ be a general position map, and let $F \colon \delprod K r \to_{\sym_r} S^{(r-1)d}$. 

Let $X \subset \delprod K r $ be a $\sym_r$-stable subcomplex such that 
\begin{itemize}
\item $\t f |_X$ avoids the diagonal $\{(x,\ldots,x) \; | \; x \in \R^d\} \subset \R^{r \times d}$,
\item $\t f |_X$ is $\sym_r$-homotopic to $F |_{X}$.
\end{itemize}
Let $\sigma_1 \times \cdots \times \sigma_r$ be a cell of $\delprod K r \setminus X$ whose boundary is contained in $X$, and let us denote by $Y$ the smallest $\sym_r$-stable subcomplex of $\delprod K r$ containing $X$ and $\sigma_1 \times \cdots \times \sigma_r$

Then there exists a map $g \colon K^m \to \R^d$ such that
\begin{itemize}
\item $\t g |_Y$ avoids the diagonal $\{(x,\ldots,x) \; | \; x \in \R^d\} \subset \R^{r \times d}$,
\item $\t g |_Y$ is $\sym_r$-homotopic to $F |_{Y}$.
\end{itemize}
\end{lemma}
\begin{proof}
Since the boundary of $\sigma_1 \times \cdots \times \sigma_r$ is contained in $X$ and $\t f|_X$ avoids the diagonal, we have,  for each $i = 1, \ldots, r$, 
\[
f^{-1} ( f\sigma_1 \cap \cdots \cap f\sigma_r ) \cap \sigma_i \subset \interior \sigma_i.
\]
Furthermore, the map
$
\t f : \boundary (\sigma_1 \times \cdots \times \sigma_r) \rightarrow S^{d(r-1)-1} 
$ is homotopic to $F$.

We are in position to apply Lemma~\ref{lem_reduction}: we find a ball $B^d$ in $\R^d$ with the three properties listed in the Lemma. Let us call $\sigma'_i$ the sub-ball in $\sigma_i$ properly embedded into $B^d$, \ie
$
\sigma'_i \overset{f}{\hookrightarrow} B^d$ , and $f \boundary \sigma_i'  = \boundary B^d \cap f \sigma'_i.
$ 

By the Combinatorial Annulus Theorem \cite[3.10]{Bryant:Piecewise-linear-topology-2002}, there exists an isotopy of $\sigma_i$ in itself that progressively retracts $\sigma_i$ to $\sigma_i'$. \Ie there exists $G^i_t : \sigma_i \rightarrow \sigma_i$ with $G^i_0$ being the identity and $G^i_1$ being a homeomorphism between $\sigma_i$ and $\sigma_i'$. We define a homotopy by
\begin{equation} \label{eq_retraction_sigma} 
\hfill
\begin{array}{rccc}
G:& \boundary ( I \times \sigma_1 \times \cdots \times \sigma_r) & \xrightarrow{fG^1 \times \cdots \times fG^r} & \R^d \times \cdots \times \R^d \setminus \delta_r \R^d\\
&(t,x_1, \ldots, x_r) & \longmapsto & (fG^1_t x_1 , \ldots fG^r_tx_r).
\end{array}
\hfill
\end{equation}
Since
\begin{equation} \label{eq_hom_class_sigma}
\hfill
\boundary ( \sigma_1 \times \cdots \times \sigma_r) \xrightarrow{f \times \cdots \times f} \R^d \times \cdots \times \R^d \setminus \delta_r \R^d
\hfill
\end{equation}
is homotopic to $F$, and $F$ is defined over $\sigma_1 \times \cdots \times \sigma_r$, therefore the homotopy class of
\[
\hfill
\boundary ( \sigma_1' \times \cdots \times \sigma_r')  \xrightarrow{f \times \cdots \times f} B^d \times \cdots \times B^d \setminus \delta_r B^d
\hfill
\]
is trivial.  Hence, we can use the first part of the Lemma~\ref{lem_whitney_trick} to find $(r-1)$ proper ambient isotopies of $B$, say $H^1_t, \ldots, H^{r-1}_t$, such that
$
H^1_1 (f \sigma_1') \cap \cdots \cap H^{r-1}_1(f\sigma_{r-1}') \cap f \sigma_r' = \emptyset.$ 
This removes the $r$-intersection set.

To finish the proof, we also need to extend the equivariant homotopy between $\t f$ and $F$ on the cell $\sigma_1 \times \cdots \times \sigma_r$, as the homotopy is already defined on $\boundary (\sigma_1 \times \cdots \times \sigma_r)$. This is when the second part of Lemma~\ref{lem_whitney_trick} becomes useful.

We define a map on 
$
\boundary ( I \times \sigma_1 \times \cdots \times \sigma_r )  \rightarrow  \R^d \times \cdots \times \R^d \setminus \delta_r \R^d
$
in the following way:
\begin{enumerate}
\item on $\{0\} \times \sigma_1 \times \cdots \times \sigma_r $, we use $F$,
\item on $[0,\frac 13] \times \boundary ( \sigma_1 \times \cdots \times \sigma_r) $, we use the homotopy from $F$ to \eqref{eq_hom_class_sigma},
\item on $[\frac 13,\frac 23] \times \boundary ( \sigma_1 \times \cdots \times \sigma_r)$, we use $G$,
\item on $[\frac 23,1] \times \boundary (\sigma_1 \times \cdots \times \sigma_r)$, we use 
$
 ( H^1_t \times \cdots \times H^{r-1}_t \times \id ) \circ ( f G^1_1 \times \cdots \times f G^r_1),
$
\item $\{1\} \times \sigma_1 \times \cdots \times \sigma_r $, we use
$
 ( H^1_1 \times \cdots \times H^{r-1}_1 \times \id ) \circ ( f G^1_1 \times \cdots \times f G^r_1).
$
\end{enumerate}
This defines a class $\theta  \in \pi_{\sum \dim\sigma_i} (S^{d(r-1)-1})$. To conclude, we need to have $\theta = 0$ (this is the condition to be able to extend to homotopy between $\t f$ and $F$).

By the second part of Lemma~\ref{lem_whitney_trick}, we can\footnote{
We can always obtain the assumption $\sigma_2\cap \cdots \cap \sigma_r \neq \emptyset$ by modifying the map $f$ as follows \cite[``Finger moves'' in the proof of Lemma~43]{MabillardWagner15}:
we pick $r-1$ spheres $S^{s_2},..., S^{s_r}$ in the interior of $B^d$ of dimension $s_2,...,s_r$ in general position and such that $S^{s_2} \cap \cdots \cap S^{s_r}$ is a sphere $S$. Then, for $i=2,...,r$, we pipe $\sigma_i'$ to $S^{s_i}$. The resulting map has the desired property.

This ``piping'' change can be absorbed by a slight modification (and renumbering) of the $H_t^i$. The support of these modifications is a collection of regular neighborhoods of $1$-polyhedra ($=$ paths used for piping).  

Also, note that the cases when, by general position, $\dim S < 0$ corresponds the trivial cases $\theta =0 $. Indeed, $\dim S <0$ corresponds to $(d-s_2) + \cdots + (d-s_r) > d$, i.e., $ (r-1)d + s_1 -d > \sum s_i $, and since $s_1-d \leq - 3$, we have $(r-1)d -1 > \sum s_i$, and so $\pi_{\sum \dim\sigma_i} (S^{d(r-1)-1}) =0$.
} perform a ``second move'' on $\sigma_1$ with an ambient isotopy $J_t$ of $B$ such that
\[
\boundary (I \times \sigma_1 \times \cdots \times \sigma_r) \xrightarrow{(J_t \times \id  \times \cdots \times  \id)\circ( H^1_1 \times \cdots \times H^{r-1}_1 \times \id ) \circ ( f G^1_1 \times \cdots \times f G^r_1)} \R^d \times \cdots \times \R^d \setminus \delta_r \R^d
\]
represents exactly $- \theta$. Therefore, by using this last move, we can assume that $\theta = 0$, \ie we can extend the equivariant homotopy between $\t f$ and $F$, as needed for the induction. 
\end{proof}

\begin{proof}[Proof of Theorem~\ref{thm_Hae_Web_an_restated}]
We are given $F: \delprod K r \rightarrow_{\sym_r} S^{d(r-1)-1}$, and we want to construct $f : K \rightarrow \R^d$ without global $r$-intersection points.

We start with a map $f : K \rightarrow \R^d$ in general position. We are going to inductively use Lemma~\ref{lem_inductive_step} to gradually remove all the global $r$-fold intersection of $f$. 

There are two levels in the induction. To describe these, let us fix a total ordering of the simplices of $K$ that extends the partial ordering by dimension, \ie
 \[
\hfill
K = \{ \tau_1 , \ldots , \tau_N\}, \qquad \dim \tau_i \leq \dim \tau_{i+1} \text{ for }1\leq i\leq N-1.
\hfill
\]

First, we give an informal plan of the ``double induction'' that we are going to use: 
we go over the list of simplices $\tau_1,...,\tau_N$, and for each simplex $\tau_i$ we consider all the global $r$-intersection of $\tau_i$ with all the simplices \emph{before} $\tau_i$ in the list.  More precisely, we consider the list $l_i$ of all $r$-tuples of pairwise disjoint simplices containing $\tau_i$ and simplices \emph{before} $\tau_i$ in the list $\tau_1,...,\tau_N$. For each $r$-tuple in $l_i$, we eliminate its global $r$-intersection points, by Lemma~\ref{lem_inductive_step}.

Therefore, once $\tau_i$ is fixed, we have a \emph{new list} $l_i$. We are going to \emph{order} $l_i$ (by a notion of dimension), and then inductively scan over it and remove the global $r$-intersections points for each $r$-tuple in $l_i$.

More formally, for the first level of the inductive argument, it suffices to prove the following: Suppose we are given a map $f: K \rightarrow \R^d$ in general position with the following two properties:
\begin{enumerate}
\item Restricted to the subcomplex
$
L = \{ \tau_1 , \ldots , \tau_{N-1}\}
$
the map $f|_L$ does not have any $r$-intersections between disjoint $r$-tuples of simplices;
\item $\t f$ restricted to $\delprod L r$ is $\sym_r$-equivariantly homotopic to $F$, where $\t f$ is the map defined in
Lemma~\ref{lem:delprod-necessary-meta}.
\end{enumerate}

Then we can redefine $f$ as to have these two properties on the whole of $K$. This is the first level of induction.

For the second level of the induction, let us define the \define{dimension} of a finite set 
of simplices as the sum of their individual dimensions. For the the purposes of this proof, we use the terminology  
\define{$k$-collection} for a set of cardinality $k$. 
Consider those $(r-1)$-collections $t$ of simplices of $L$ that, together with $\tau_N$, form an $r$-collection of pairwise disjoint simplices.
We fix a total ordering of these $(r-1)$-collections that extends the partial ordering given by dimension, \ie we list them as
$$t_1 , \ldots , t_M,$$ with 
$\dim t_i \leq \dim t_{i+1} \text{ for }1\leq i<M.$
(Thus, each $t_i$ is an $(r-1)$-collection of simplices of $L$, and $t_i$ joined with $\tau_N$ is a $r$-collection of pairwise disjoint simplices.)
Once again, inductively, it suffices to prove the following: Assuming that $f$ has the two properties
\begin{enumerate}
\item For each $(r-1)$-collection $t_i$ in the list
$
t_1 , \ldots , t_{M-1},
$
the map $f$ does not have any $r$-intersection with preimages in the $r$-collection formed by adjoining $\tau_N$ to $t_i$.

\item the map $\t f$ is $\sym_r$-equivariantly homotopic to $F$ on the complex
\[
\hfill
\delprod L r \cup \bigcup_{i \leq M-1} [ t_i \cup \{ \tau_N  \}] \subseteq \delprod K r,
\hfill
\]
where the operator $[-]$ converts an unordered $r$-collection of pairwise disjoint simplices of $K$ into the set of its corresponding cells\footnote{
E.g., 
$
[\{ \alpha , \beta ,\gamma\}] = \{ \alpha \times \beta \times \gamma , \alpha \times \gamma \times \beta, \beta \times \alpha \times \gamma , \beta \times \gamma \times \alpha, \gamma \times \alpha \times \beta, \gamma \times \beta \times \alpha  \}
$.
}
in $\delprod K r$.
\end{enumerate}
Then we can modify $f$ as to have these two properties on the lists $t_1 , \ldots , t_M$.

This inductive step is directly implied by Lemma~\ref{lem_inductive_step}.
\end{proof}

%% file: proof-of-the-second-lemma.tex
\section{The Proof of the Local Disjunction Lemma} 
\label{s:proof-second-lemma}

Our goal in this section is to prove Lemma~\ref{lem_whitney_trick}, which was used in the previous section to prove the sufficiency of the deleted product criterion in the $r$-metastable range.

Throughout this section, we assume that $m,d,r$ are positive integers satisfying \eqref{eq_metastable}. Furthermore, we will denote the sum of the dimensions $s_i$ of the balls $\sigma_i$ by
$$s:=s_1+\ldots +s_r.$$

The proof of Lemma~\ref{lem_whitney_trick} is essentially inductive: we reduce from $r$ balls to $(r-1)$ balls. The trick is to consider the intersection pattern of the first $(r-1)$ balls $\sigma_1 , \ldots , \sigma_{r-1}$ on $\sigma_r$. If each of the intersections $\sigma_i \cap \sigma_r$, $1\leq i\leq r-1$, were a ball properly embedded in $\sigma_r$, then we could solve the situation first at the level of $\sigma_r$ (\ie  remove the $(r-1)$-intersections between the $\sigma_i \cap \sigma_r$), and then extend the solution to $B$, thus completing the induction.

However, the intersections $\sigma_i \cap \sigma_r$ need not be balls, so our first task is to move $\sigma_1 , \ldots , \sigma_{r-1}$ inside $B$ as to modify their intersection with $\sigma_r$. As it will turn out, if we manage to increase sufficiently the connectedness of the intersections $\sigma_i \cap \sigma_r$, then Theorem~\ref{thm_Zeeman} becomes useful to reduce the situation (as in the proof of Lemma~\ref{lem_reduction}) in such a way that the intersections  
$\sigma_i \cap \sigma_r$ do become balls. For this to work, $\sigma_i \cap \sigma_r$ needs to be $ \dim(\sigma_1 \cap \cdots \cap \sigma_r) $-connected.

\input{proof-of-the-second-lemma-connectivity}
\input{proof-of-the-second-lemma-balls}
\input{proof-of-the-second-lemma-complete}

%% file: proof-of-the-second-lemma-connectivity.tex
\subsection{Increasing the connectivity of the intersections}
\label{sec_increasing_the_connectivity}

\begin{proposition} \label{prop_surg}
With the same notations as in Lemma~\ref{lem_whitney_trick}, for each $i=1,\ldots,{r-1}$, there exists a proper ambient isotopy $H_t$ of $B$ such that 
$H_1(\sigma_i) \cap  \sigma_r$  is $\dim(\sigma_1 \cap \cdots \cap \sigma_r)$-connected,
and such that 
\begin{equation}\label{eq_surgery_homotopy_class_well_defined}
I \times \boundary (\sigma_1 \times \cdots \times \sigma_i \times \cdots \times \sigma_r) \xrightarrow{\incl \times \cdots \times H_t \times \cdots \times \incl} (B^d \times \cdots \times B^d) \setminus \delta_r(B^d)
\end{equation}
is well-defined, \ie  its image is disjoint from the diagonal $\delta_r(B^d)$.
\end{proposition}

\begin{proof} Proposition~\ref{prop_surg} follows directly by inductively using the Lemma~\ref{lem_surgey_step} (below), as in \cite[Lemma~2]{Milnor}.
\end{proof}

\begin{lemma}\label{lem_surgey_step}
\begin{enumerate}[(a)]
\item 
With the same notation as above, for all $
1 \leq k \leq \dim(\sigma_i \cap \cdots \cap \sigma_r)$
and $S^k \rightarrow \sigma_i \cap \sigma_r$ representing a homotopy class in $\pi_k(\sigma_i \cap \sigma_r)$, 
there exists a proper ambient isotopy $H_t$ of $B$ such that, for $j < k$,
\[
\pi_j ( H_1 ( \sigma_i) \cap \sigma_r) \iso \pi_j(\sigma_i \cap \sigma_r),
\]
and 
\[
\pi_k ( H_1 ( \sigma_i) \cap \sigma_r) \iso \pi_k (\sigma_i \cap \sigma_r) / \mbox{a subgroup containing $[S^k]$}. 
\]
\item 
An analoguous statement holds for $k=0$: If $\sigma_i\cap \sigma_r$ has more than one connected component, then there exists a proper ambient isotopy $H_t$ of $B$ such that $H_1(\sigma_i) \cap \sigma_r$ has one less connected component.
\end{enumerate}
In both cases (a) and (b) with have the following additional property of $H_t$:  the map~\eqref{eq_surgery_homotopy_class_well_defined} defined using $H_t$ avoids the diagonal.
\end{lemma}

\begin{figure}
\begin{center}
\noindent  \begin{minipage}{0.45\textwidth}
\includegraphics[scale=1]{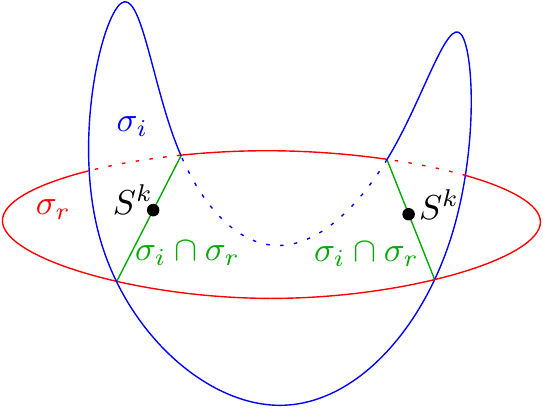}
\caption{$S^k$ represents a non-zero element of the homotopy group $\pi_k(\sigma_i \cap \sigma_r)$.}
\label{fig_surgery_beginning}
\end{minipage}
\hfill
\begin{minipage}{0.45\textwidth}
\includegraphics[scale=1]{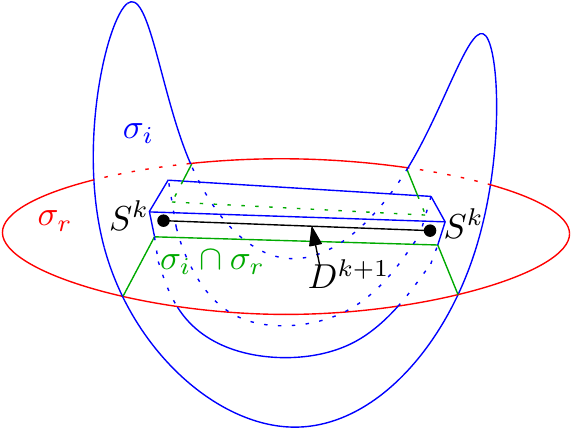}
\caption{By moving a sub-ball of $\sigma_i$ inside of $B$, we modify the intersection of $\sigma_i$ and $\sigma_r$ as to ``kill'' by surgery the homotopy class represented by $S^r \subseteq \sigma_i \cap \sigma_r$.}
\label{fig_surgery_end}
\end{minipage}
\end{center}
\end{figure}

\begin{figure}
\begin{center}
\includegraphics[scale=1]{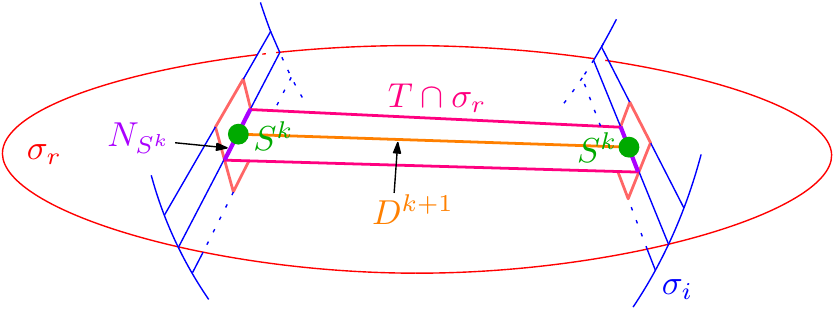}
\caption{The different steps in the construction of the handle used for the ambient surgery.}
\label{fig_surgery_tube}
\end{center}
\end{figure}

Here, we only present the proof of the part (a), i.e., for $k\ge 1$. For $k=0$, the construction is similar, and is aready presented in \cite{MabillardWagner15} as \emph{piping and unpiping}.

Our main technique in the proof is to use \emph{surgery} (as presented by Milnor \cite{Milnor}) to increase the connectivity of $\sigma_i \cap \sigma_r$. The precise definition of surgery used in our situation is given later (Definition~\ref{def_amb_surg}).

Figure~\ref{fig_surgery_beginning} illustrates the situation, and Figure~\ref{fig_surgery_end} tries to illustrate how we intend to `kill' a homotopy class of $S^k \in \pi_l(\sigma_i \cap \sigma_r)$ by surgery.

\begin{remark}
We decompose the proof of Lemma~\ref{lem_surgey_step} into a series of Lemmas. 

For the first two Lemmas, we need to use a PL analogous of vector bundles for smooth manifold. In the PL category, this analogous notion is called \emph{block bundles}. 
We review the part of this theory that we need in Appendix~\ref{chap_block_bundles}. (See also \cite{Bryant:Piecewise-linear-topology-2002} for a rapid introduction, or the original \cite{BBI}). We will need results from \cite{BBI,BBII}.

Since we only work in the PL category, we sometimes only say \emph{bundle} instead of \emph{block bundle}.
\end{remark}

First we render, once and for all, the intersections transverse: 

\begin{lemma}\label{sublem_trivial_bundle_intersection}
With the same notations as in Lemma~\ref{lem_whitney_trick},
we can assume that $\sigma_r$ is unknotted in $B^d$, i.e.,
\[
B^d = \sigma_r \times [-1, 1]^{d-s_r},
\]
and we can also assume that $\sigma_i$ intersects $\sigma_r$ transversely (Definition~\ref{def_transversality}), i.e., for $\epsilon>0$ small enough,  $\sigma_r\times \epsilon [-1,1]^{d-s_r}$ is a normal block bundle to $\sigma_r$ in $B^d$, and we have
\[
\sigma_i \cap  ( \sigma_r\times \epsilon [-1,1]^{d-s_r} )  = (\sigma_i \cap \sigma_r) \times \epsilon [-1,1]^{d-s_r}.
\] 
\end{lemma}
\begin{proof}
The first statement follows from Zeeman's Unknotting of balls. The second statement follows by Theorem~\ref{thm_transversality}: there exists an $\epsilon$-isotopy of $B$ carrying $\sigma_i$ locally transverse to $\sigma_r$. Using a collar on $\boundary B$, we can furthermore assume that this isotopy is fixed on  $\boundary B$.
\end{proof}

\begin{remark}
In the sequence of lemmas that follows, $S^k \rightarrow \sigma_i \cap \sigma_r$ represents an homotopy class in $\pi_k(\sigma_i \cap \sigma_r) $, which we want to ``kill''.
\end{remark}

\begin{observation} We have
\[
2 \dim (S^k) + 1 \leq \dim (\sigma_i \cap \sigma_r).
\] 
Indeed, this is the case if
\[
2 k + 1 \leq s_i + s_r -d 
\]
So we have to show
\[
k < \frac{s_i + s_r-d}2
\]
Since $k \leq s - (r-1) d $, it is sufficient to show 
\[
s - (r-1) d
< \frac{s_i + s_r -d } 2
\]
and, after rearrangement and using $s_i \leq m$, we get the sufficient condition
\[
2(r-1)m < (2r -3)d,
\quad
\text{
which is the case since}
\quad
\frac{2(r-1)}{2r-3} m \leq \frac{r+1}r m < d,
\]
where the first inequality is true for $r \ge 3$, and the second follows from the metastable range.
\end{observation}

\begin{lemma}\label{sublem_trivial_bundle_S_k}
In the situation given by Lemma~\ref{sublem_trivial_bundle_intersection},
let $a: S^k \rightarrow \sigma_i \cap \sigma_r$ represents an homotopy class in $\pi_k(\sigma_i \cap \sigma_r) $. Then there exists an embedded copy of $S^k \subset \sigma_i \cap \sigma_r$ such that its inclusion map is homotopic to $a$, and with  the two additional properties: 
\begin{enumerate}[(1)]
\item  the normal block bundle of $S^k \subset \sigma_i \cap \sigma_r$ is trivial.
\item 
Let $N_{S^k}$ be a regular neighborhood of $S^k$ inside $\sigma_i \cap \sigma_r$.  Then 
\[
N_{S^k} \iso S^k \times [-1,1]^{ s_i + s_r - d -k  },
\]
containing $S^k$ as $S^k \times 0 $.
\end{enumerate}
\end{lemma}
\begin{proof}
The existence of the embedded copy of $S^k$ follows by general position (and the above observation).

The first property follows from Theorem~\ref{prop_s_k_in_M_pi_manifold} from Appendix~\ref{chap_block_bundles}: we only need to check that the tangent bundle of $\sigma_i \cap \sigma_r$ is stably trivial. To see it: let us consider the normal bundle of $\sigma_i \cap \sigma_r$ in $\sigma_i$ which is isomorphic to $\epsilon^{d-s_r}$ (by hypothesis), hence (using notations defined in Appendix~\ref{chap_block_bundles} and Theorem~\ref{thm_tangent_normal})
\[
t(\sigma_i \cap \sigma_r) \oplus \epsilon^{d-s_r} 
= t(\sigma_i) | (\sigma_i \cap \sigma_r ) 
= \epsilon^{i} /(\sigma_i \cap \sigma_r ).
\]
I.e., $t(\sigma_i \cap \sigma_r) / ( \sigma_i \cap \sigma_r)$ is stably trivial. 

The second property about $N_{S^k}$ follows by the correspondence between regular neighborhoods and normal block bundles (Theorem~\ref{thm_corr_regular_neighborhood_bb} in Appendix~\ref{chap_block_bundles}).
\end{proof}

\begin{lemma}\label{sublem_unknot_ribbon}
In the situation given by Lemma~\ref{sublem_trivial_bundle_S_k}, with $S^k \subset \sigma_i \cap \sigma_r$ and the two additional properties.
There exists a ball $D^{k+1}$ in $\sigma_r$ with 
\[
D^{k+1} \cap \sigma_i = \boundary D^{k+1} = S^k,
\]
and which avoids the other $\sigma_j$.

Furthermore, the trivialisation of $N_{S^k}$ can be extended to $D^{k+1}$, i.e., there exists in $\sigma_r$ 
\begin{equation} \label{eq_def_iso_D}
N_{D^{k+1}} \iso D^{k+1} \times [-1,1]^{ s_i + s_r - d -k  }
\end{equation}
containing $D^{k+1}$ as $D^{k+1} \times 0$ and with 
\[
N_{D^{k+1}} \cap \sigma_i = N_{S^k} \iso S^k \times [-1,1]^{ s_i + s_r - d -k  },
\]
and this last homeomorphism is the restriction of \eqref{eq_def_iso_D}.
\end{lemma}

\begin{remark}
For proving the second part of Lemma~\ref{sublem_unknot_ribbon}, we could use the PL-analogue of  Stiefel manifolds \cite[p.~274]{BBII}: the obstruction to extending the  trivialisation of $S^k$ is always trivial in the metastable range. But, to avoid entering more deeply into the theory of block bundles, we rather use the following unknotting theorem of Hudson\footnote{This step follows a proof of A. Skopenkov in \cite[p.~9]{skopenkov2005new}.}:

\begin{theorem}[{\cite[Unknotting Theorem Moving the Boundary, 10.2, p.~199]{Hudson}}] \label{thm_Hudson_unknotting_moving_boundary}
If $f,g :M^m \rightarrow Q^q$ are proper PL embeddings between manifolds $M$ and $Q$. 
Then $f,g$ homotopic as maps of pairs $(M, \boundary M) \rightarrow (Q, \boundary Q)$ implies that $f,g$ are ambient isotopic provided that 
\begin{itemize}
\item $M$ is compact
\item  $q-m \ge 3$
\item $(M,\boundary M)$ is $(2m-q+1)$-connected
\item $(Q,\boundary Q)$ is $(2m-q+2)$-connected
\end{itemize}
\end{theorem}
\end{remark}
A consequence of Hudson Theorem is
\begin{corollary}
\label{cor_hudson_unknotting_ribbons}
Let $f,g : S^{k} \times [-1,1]^{m-k} \embeds B^q$ be two proper embeddings, then $f$ and $g$ are ambient isotopic, provided
\[ 
q \ge m+k+3.
\]
\end{corollary}
\begin{proof}
Since $q-m \ge k+3 \ge 3$, and $(B^q, \boundary B^q)$ is clearly sufficently connected, we only need to analyse the connectivity of the pair
\[
(S^k \times [-1,1]^{m-k} , S^k \times S^{m-k-1})
\]
which we need to be $ ( 2m-q+1)$-connected. Let us consider the exact sequence in homotopy for this pair
\[
\cdots \rightarrow
\pi_i( \underbrace{ S^k \times S^{m-k-1}}_{:=\boundary X}) 
\rightarrow \pi_i( \underbrace{S^k \times [-1,1]^{m-k}}_{:=X}) 
\rightarrow \pi_i( X, \boundary X)
\rightarrow \cdots
\]
For $i<m-k-1$, the above sequence can be rewritten as
\[
\cdots \rightarrow
\pi_i( S^k ) 
\rightarrow \pi_i(S^k )
\rightarrow \pi_i( X ,\boundary X)
\rightarrow \cdots
\]
Since the $\pi_i( S^k ) 
\rightarrow \pi_i(S^k )$ is an isomorphism, we get $\pi_i(X, \boundary X) =0$ as long as $i < m-k-1$. So we are left with checking 
\[
2m -q + 1 < m-k -1,
\]
I.e., $m+k+2 < q$.
\end{proof}

We can now proceed to the proof of Lemma~\ref{sublem_unknot_ribbon}:
\begin{proof}[Proof of Lemma~\ref{sublem_unknot_ribbon}]
The first statement follows by general position and the metastable range hypothesis.

To prove the existence of $N_{D^{k+1}}$, let us first take a regular neighborhood $V$ of $D^{k+1}$ in $\sigma_r$. We can assume that 
\[
V \cap \sigma_i = N_{S^k} \iso S^k \times [-1,1]^{s_i+s_r-d-k}.
\]
If $N_{S^k}$ unknots in $V$ (in the sense of Theorem~\ref{thm_Hudson_unknotting_moving_boundary}), then the existence of $N_{D^{k+1}}$ is immediate: we use an ``standard'' version of $N_{S^k}$ to construct $N_{D^{k+1}}$, and move it to our situation by the isotopy given by the unknotting theorem.

So we only need to check the hypothesis of Corollary~\ref{cor_hudson_unknotting_ribbons}. For us here, $q:=s_r$ and $m:= s_i+s_r-d$, so we need
\[
s_r \ge (s_i+s_r-d) + k + 3.
\]
which reduces to $d-s_i -3 \ge k$, which is true if $d -s_i -3 \ge  s - (r-1)d$, and this is implied by $rd \ge (r+1) m +3$, i.e., the metastable range hypothesis.
%
\end{proof}

\begin{lemma}[Existence of the surgery-handle]\label{sublem_surgery_tube}
In the situation given by Lemma~\ref{sublem_unknot_ribbon}, there exists in $B$ a handle 
\[
T := D^{k+1} \times [-1,1]^{s_i + s_r -d-k} \times [-1,1]^{d-s_r}
\]
such that 
\begin{itemize} 
\item 
$T$ contains $D^{k+1}$ as $D^{k+1} \times 0$,
\item 
$T$ intersects $\sigma_r$ as $D^{k+1} \times [-1,1]^{s_i + s_r -d-k} \times 0$, 
\item 
$T$ intersects $\sigma_i$ as $S^{k}\times [-1,1]^{s_i + s_r -d-k} \times [-1,1]^{d-s_r}$.
\end{itemize}
Figure~\ref{fig_surgery_tube} illustrates the handle $T$.
\end{lemma}
\begin{proof}
This follows from the construction of $D_{k+1}$ (Lemma~\ref{sublem_unknot_ribbon}) and the transversality of the intersection of $\sigma_i$ and $\sigma_r$ (Lemma~\ref{sublem_trivial_bundle_intersection}).
\end{proof}

\begin{definition}[Ambient surgery] \label{def_amb_surg}
Let $S^k$ be an embedded sphere in (the interior of) $\sigma_i$ with a trivialized regular neighborhood $ S^k \times [-1,1]^{s_i - k}$, and let $T$ be a \define{handle based on $ S^k$}, i.e.,
\[
 T =  D^{k+1} \times [-1,1]^{s_i-k} \subseteq B^d
\]
for a ball $ D^{k+1}$ with 
\[
 T \cap \sigma_i = \boundary D^{k+1} \times [-1,1]^{s_i-k}
=  S^{k+1} \times [-1,1]^{s_i-k}.
\]

Using $T$, we perform a \define{ambient surgery} on $\sigma_i \subset B^d$ by constructing the new manifold 
\begin{equation} \label{eq_surg_on_sigma_i}
\sigma_i^* :=  ( \sigma_i \setminus (S^k \times [-1,1]^{s_i - k} )) \cup ( D^{k+1} \times \boundary [-1,1]^{s_i -k}) \subset B^d
\end{equation}
\end{definition}

In order to attach the handle $T$ we made several choices: the choice of $S^{k}$, its regular neighborhood $S^{k} \times [-1,1]^{s_i -k}$, the `core' $D^{k+1}$, etc. In the next Lemma, we show that, up to isotopy, there is only one way to attach a handle $T$ to $\sigma_i$:
\begin{lemma}
If $ S^k$ and $\t S^k$ are embedded spheres in $\sigma_i$ with a trivialized regular neighborhoods and handles $T$ and $\t T$ as in Definition~\ref{def_amb_surg}.

Then performing a surgery on $\sigma_i$ using $T$ or $\t T$ produces two homeomorphic manifolds $\sigma_i^*$ and $\t \sigma_i^*$ that are connected by a proper ambient isotopy of $B^d$. 
\end{lemma}
\begin{proof}
By Irwin's Theorem \cite[Ch.~VIII, Theorem~24]{Zeeman:Seminar-on-combinatorial-topology-1966}, there exists a proper isotopy of $\sigma_i$ ``throwing'' $\t S^k$ onto $S^k$, so we can assume $S^k = \t S^k$ (since this isotopy can be extended to $B^d$ \cite{Hudson:Extending-piecewise-linear-isotopies-1966}).

Then, by the uniqueness of regular neighborhoods, we can assume that the trivialisation of the normal block bundle are identical \cite[Theorem~4.4]{BBI}. 

We have reached the situation where $T \cap \sigma_i = \t T \cap \sigma_i$.

Let us take a cone $C$ in $B^d$ over $D^{k+1} \cup \t D^{k+1}$. By general position\footnote{We have to check, e.g., $(k+2) + s_r -d <0$, i.e., $k < d -s_r-2$, which is true if $s -(r-1)d < d-s_r-2$, and this is implied by the metastable hypothesis $(r+1)m + 3 \leq r d$.}, this cone avoids $\sigma_i$ (except on $S^k = \t S^k$). Let us take a regular neighborhood $V$ of the collapsible space $C$ in $B^d$ \textit{relative} to $S^k = \t S^k$. Hence, $V$ is a $d$-ball, and 
\[
V \cap \sigma_i = S^k, \quad \mbox{\cite[p.~719, (iii)]{HudsonZeemanRN1}}.
\]
Inside of $V$, we can find an ambient isotopy (fixed on the boundary) `throwing' $\t D^{k+1}$ to $D^{k+1}$, hence, we can assume that $D^{k+1} = \t D^{k+1}$.

We have reached the situation where both $T$ and $\t T$ are equal on  $\sigma_i$ and have the same `core' $D^{k+1} = \t D^{k+1}$.

To conclude, let us take a regular neighborhood $N$ of $D^{k+1}$. We can assume that 
\begin{itemize}
\item $N \cap \sigma_i = S^k \times [-1,1]^{s_r-k} = T \cap \sigma_i = \t T \cap \sigma_i$, and
\item $T, \t T \subseteq N$ (after, possibly, shrinking the handles).
\end{itemize}
We have that 
\[
\sigma_i^* \cap N
=
D^{k+1} \times [-1,1]^{s_i-k}
\iso 
\t \sigma_i^* \cap N
\]
and
\[
\boundary  ( \sigma_i^* \cap N )
= \boundary (\t \sigma_i^* \cap N)
= S^{k} \times \boundary [-1,1]^{s_i-k}.
\]

Hence, to conclude, we only have to check that $D^{k+1} \times \boundary [-1,1]^{s_i-k} $ unknots inside of $N$ (keeping the boundary fixed). First, we observe that two proper maps $D^{k+1} \times S^{s_i-k-1} \rightarrow B^d$ that are equal on the boundary are always homotopic (by a straight-line homotopy). Hence, by Irwin's Theorem \cite[Ch. VIII, Theorem 24]{Zeeman:Seminar-on-combinatorial-topology-1966}, we only need to check
\[
2 s_i -d + 1 \leq s_i -k +1, \quad \mbox{i.e.,} \quad k+3 \leq d-s_i
\]
which is true if  $r m - (r-1)d + 3 \leq d -s_i$, and this is implied by $(r+1) m + 3 \leq rd$ (the metastable range hypothesis).
\end{proof}

\begin{figure}
\begin{center}
\includegraphics[scale=1]{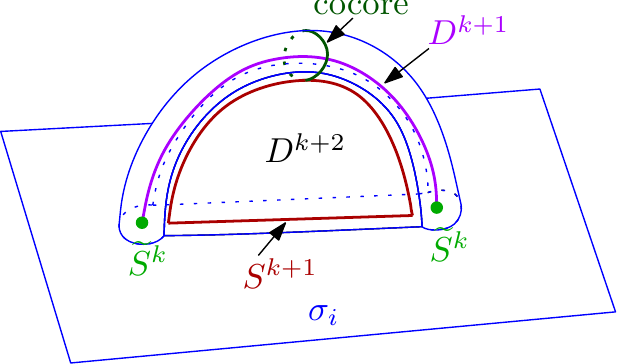}
\caption{We perform two complementary surgeries on $\sigma_i$ such that the resulting manifold $\sigma_i'$ is a ball homeomorphic to $\sigma_i$.}
\label{fig_double_surgery}
\end{center}
\end{figure}

\begin{lemma}[Existence of a complementary handle]
Let $S^k$ and $T$ be as in Definition~\ref{def_amb_surg}.

Then  there exists an handle (see Figure~\ref{fig_double_surgery})
\[
T^c = D^{k+2} \times [-1,1]^{s_i - k - 1} \subseteq B^d
\]
with 
\begin{itemize}
\item $T^c \cap \sigma_i^* = S^{k+1} \times [-1,1]^{s_i -k-1}$ for a $(k+1)$-sphere $S^{k+1}$ such that 
\item 
$S^{k+1}$ intersects the cocore of $T$ 
\[
0^{k+1} \times \boundary [-1,1]^{s_r - k} \subseteq \boundary T
\]
at exactly one point.
\item Furthermore, we can assume that $S^{k+1}$ is at positive distance of $\sigma_r$. 
\end{itemize} 
\end{lemma}
\begin{proof}
By the previous lemma, there exists, up to proper isotopy of $B^d$, an unique way to to perform a surgery by the handle $T$ on $\sigma_i$. From this fact, the existence of the complementary handle $T^c$ is immediate.

For the last property, we need to shift $S^{k+1}$ to general position\footnote{We want $(k+1) + (s_i+ s_r -d) - s_i <0$, i.e., $k < d-s_r-1$. But $k \leq s -(r-1)d$, so we only need $s -(r-1)d < d-s_r-1$, i.e., $s + s_r + 1 < r d$, and this is true in the metastable range $(r+1)m + 3 \leq r d$.}.
\end{proof}

\begin{remark}
We call $T^c$ the `complementary handle' to $T$.
\end{remark}

\begin{lemma}\label{sub_complementary_tube}
Let $\alpha \in \pi(\sigma_i\cap \sigma_r)$. Then there exists a sphere $S^k \subset \sigma_i\cap \sigma_r$ and a handle $T$ as in Defintion~\ref{def_amb_surg} such that 
performing a surgery on $\sigma_i$ by the handle $T$, followed by a surgery by the handle $T^c$ produces a manifold $\sigma_i^{**}$ which is a $s_i$-ball. Furthermore for $j<k$
\[
\pi_j(\sigma_i^{**}\cap \sigma_r) \iso \pi_j(\sigma_i \cap \sigma_r)
\] 
and
\[
\pi_k(\sigma_i^{**} \cap \sigma_r) \iso \pi_k(\sigma_i \cap \sigma_r)/\mbox{a subgroup containing $\alpha$}.
\]
\end{lemma}
\begin{proof}
The existence of $S^k$ is given by Lemma~\ref{sublem_trivial_bundle_S_k}. The existence of $T$ is given by Lemma~\ref{sublem_surgery_tube}. The existence of $T^c$ is given by Lemma~\ref{sub_complementary_tube}.

To conclude, one notices
\begin{itemize}
\item 
By the first surgery using the handle $T$, we have `killed' the homotopy class $\alpha = [S^k] \in \pi_k(\sigma_i \cap \sigma_r)$, i.e., by construction,
\[
\sigma_i^* \cap \sigma_r
=  (  ( \sigma_i \cap \sigma_r) \setminus  (S^k \times [-1,1]^{s_i + s_r-d - k} \times 0)) \cup (D^{k+1} \times \boundary [-1,1]^{s_i + s_r-d - k} \times 0)
\]
and so we have killed $[S^k]$ in the sense of \cite[Lemma~2]{Milnor}.
\item 
The effect of two surgeries by complementary handles cancels, hence $\sigma_i^{**}$ is a $s_i$-ball \cite[Lemma~6.4]{Rourke:Introduction-to-piecewise-linear-topology-1982}.  \qedhere
\end{itemize}
\end{proof}

\begin{proof}[Proof of Lemma~\ref{lem_surgey_step}]
One combines Lemma~\ref{sub_complementary_tube} with Zeeman's Unknotting of balls.
\end{proof}

%% file: proof-of-the-second-lemma-balls.tex
\subsection{Proof of Lemma~\ref{lem_whitney_trick} for balls}

\begin{proposition} \label{prop_case_balls}
The first part of Lemma~\ref{lem_whitney_trick} is true if we add the following hypothesis: for each $i=1 , \dots, r-1$,
\[
\sigma_i \cap \sigma_r \text{ is a $(s_i + s_r-d)$-ball properly contained in $\sigma_r$.}
\]
\end{proposition}

Before proving Proposition~\ref{prop_case_balls}, we need two Definitions and two Lemmas.

\begin{definition}\label{def_suspended_map}
Let $g \colon \sigma_1 \sqcup \, \cdots \, \sqcup \sigma_r \to B^d$ be balls properly mapped inside $B^d$, with the dimensional restriction of Lemma~\ref{lem_whitney_trick}. 

We say that $g$ is a \define{suspended map} if it has the following structure
\begin{itemize}
\item $g \sigma_r$ is an embedded and unknotted ball inside $B^d$, hence we can assume that
\[
B^d = ( g \sigma_r) * S^{d - s_r -1},
\]
for some $S^{d - s_r -1}$.
\item For $i=1,...,r-1$, the preimage by $g|_{\sigma_i}$ of $g \sigma_r\subset B^d$ is a ball properly embedded and  unknotted inside $\sigma_i$. I.e.,
\[
\sigma_i =  g|_{\sigma_i}^{-1} ( g \sigma_r )  * S^{d-s_r-1},
\]
for some $S^{d - s_r -1}$.

\textbf{Notation}: $\tau_i := g|_{\sigma_i}^{-1} ( g \sigma_r )\subset \sigma_i$.

\item For $i=1,...,r-1$, $g$ is defined as follows: 
\begin{itemize}
\item 
the sphere $S^{d-s_r-1} \subset \sigma_i$ is mapped homeomorphically to $S^{d-s_r-1} \subset B^d$,
\item the ball
$\tau_i \subset \sigma_i$ is 
properly map to $g \sigma_r$.
\item $g$ is defined elsewhere on $\sigma_i$ by interpolating in the obvious way between the two joins
\[
\sigma_i =  \tau_i  * S^{d-s_r-1}
\quad
\mbox{and}
\quad
B^d = ( g\sigma_r) * S^{d - s_r -1}.
\]
\end{itemize}
\end{itemize} 
Figure~\ref{fig_suspending_maps} shows on the right a suspended map.
\end{definition}

\begin{figure}
\begin{center}
\includegraphics[scale=1]{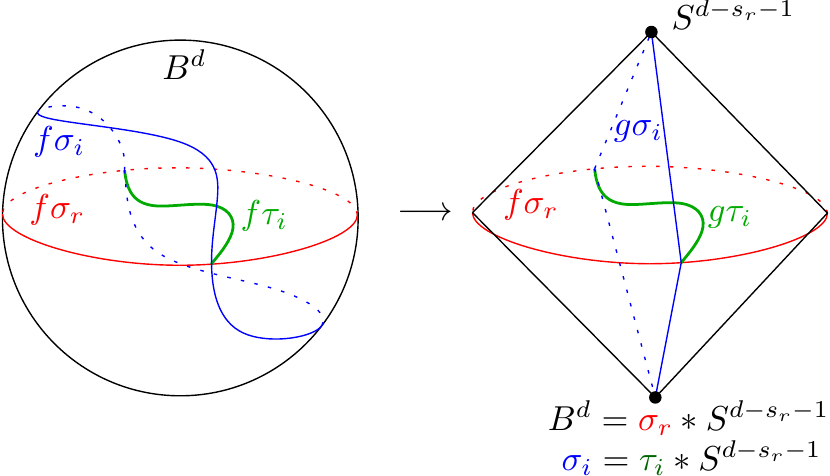}
\caption{Using that $\sigma_r$ unknots in $B^d$ and that $\sigma_i \cap \sigma_r$ unknots in $\sigma_r$, we change the setting to a suspension over $\sigma_r$.}
\label{fig_suspending_maps}
\end{center}
\end{figure}

\begin{lemma}[Suspended maps, Figure~\ref{fig_suspending_maps}]\label{app_suspended_maps}
 Let $f \colon \sigma_1 \sqcup \, \cdots \, \sqcup \sigma_r \to B^d$ be balls properly embedded inside $B^d$ in general position, with the dimensional restriction of Lemma~\ref{lem_whitney_trick}, and with the additional hypothesis of Proposition~\ref{prop_case_balls}, i.e., 
for each $i=1 , \dots, r-1$,
$
\sigma_i \cap \sigma_r 
$
is a $(s_i + s_r-d)$-ball properly contained in $\sigma_r$.

Then there exists a suspended map $g \colon \sigma_1 \sqcup \, \cdots \, \sqcup \sigma_r \to B^d $, such that
\begin{itemize}
\item 
the intersection classes of $f$ and $g$ are equal
\item 
$f|_{\sigma_r} = g|_{\sigma_r}$
\item For $i= 1,..., r-1,$ we have
$f|_{\sigma_i}^{-1}(\sigma_r) = g|_{\sigma_i}^{-1}(\sigma_r)=:\tau_i$. 
\item 
$f|_{\tau_i} = g|_{\tau_i}$.
\end{itemize}
\end{lemma}
\begin{proof}
To simplify notation, during the proof we assume that $f$ is an inclusion map, i.e., $\sigma_i \subset B^d$.

The existence of $g$ will follow from the facts that
\begin{itemize}
\item 
$\sigma_r$ unknots in $B^d$, 
\item 
$\sigma_i \cap \sigma_r$ unknots inside of $\sigma_r$,
\item 
the modifications applied during the unknotting on $\sigma_1,\ldots ,\sigma_{r-1}$ do not change the homotopy class that we are interested into. 
\end{itemize}
More precisely, since $\sigma_i \cap \sigma_r$ unknots inside of $\sigma_r$, we can represent $\sigma_r$ as 
\[
\sigma_r = (\sigma_i \cap \sigma_r) * S^{ d -s_i -1}, \quad \text{and so}\quad B^d =  (\sigma_i \cap \sigma_r) * S^{ d -s_i -1} * S^{d-s_r-1}
\] 
Hence, we define a retraction from 
\[
 B^d   \setminus   ( \emptyset * S^{ d -s_i -1}* \emptyset ) \quad \text{onto} \quad (\sigma_i \cap \sigma_r)* \emptyset * S^{d-s_r-1},
\]
and, using this retraction on $\sigma_i$, we can assume that
$\sigma_i \subseteq   (\sigma_i \cap \sigma_r)*  S^{d-s_r-1}$.

If $B^{d -s_i}$ is the ``standard ball'' in $\sigma_r$ with boundary $S^{ d -s_i -1}$, then $\sigma_i \cap \sigma_r$ intersects this ball precisely once, and this translates into the fact that $\boundary \sigma_i$ is a generator of the homotopy group $\pi_{s_i-1} (\boundary (\sigma_i \cap \sigma_r)*  S^{d-s_r-1})\iso \Z$.

Hence, we can assume that $\sigma_i = (\sigma_i \cap \sigma_r)*  S^{d-s_r-1}$, after an homotopy of $\sigma_i $  inside of $(\sigma_i \cap \sigma_r)*  S^{d-s_r-1}$ (keeping $\boundary \sigma_i$ on $\boundary B^d$).
\end{proof}

\begin{lemma}[Commuting Square for Suspended Maps]\label{lem_commuting_square}
Let $f \colon \sigma_1 \sqcup \, \cdots \, \sqcup \sigma_r \to B^d$ be a suspended map. 
Then there exists a diagram commuting up to homotopy
\begin{equation} \label{eq_diam_simple}
\hfill
\begin{gathered}
\xymatrix{
\boundary(\sigma_1 \times \cdots \times \sigma_r) \ar[d]
\ar[r]^-{\homeq} 
&
\Sigma^{d(r-1) - s_r(r-2)} \boundary(\tau_1  \times \cdots \times \tau_{r-1} ) \ar[d]   \\
B \times \cdots \times B  \setminus \delta_r(B)   \ar[r]^-{\homeq} &
\Sigma^{d(r-1) - s_r(r-2)} ( (\sigma_r \times \cdots \times  \sigma_r ) \setminus \delta_{r-1}(\sigma_r) ) 
}
\end{gathered}
\hfill
\end{equation}
where
\begin{itemize}
\item 
the map on the left is the obvious one, representing an element $\alpha \in \pi_{s-1}(S^{d(r-1)-1})$, 
\item 
the map on the right is the suspension $\Sigma$ applied $( d(r-1) - s_r(r-2) )$ times to the map
\[
\hfill
 \boundary(\tau_1 \times \cdots \times \tau_{r-1})  \rightarrow  (\sigma_r \times \cdots \times  \sigma_r)  \setminus \delta_{r-1}(\sigma_r) ,
\hfill
\]
and this map represents an element
\[
\beta \in \pi_{ (s_1 + s_r -d) + \cdots + (s_{r-1} + s_r  -d) -1 }(S^{s_r(r-2) -1}).
\]
\item 
The two horizontal maps are defined within the proof. 
\end{itemize}
\end{lemma}
We defer the proof of Lemma~\ref{lem_commuting_square} to Section~\ref{sec_construction_square}.

\begin{proof}[Proof of Proposition~\ref{prop_case_balls}]
We apply Lemma~\ref{app_suspended_maps}, to get a suspended map $ f\colon \sigma_1 \sqcup \, \cdots \, \sqcup \sigma_r \to B^d$ with the same intersection class as our initial map. 

From Diagram \eqref{eq_diam_simple} in Lemma~\ref{lem_commuting_square}
\[
\Sigma^{d(r-1) - s_r(r-2)} \beta = \alpha.
\]
But $\alpha= 0$, and we are in the stable range of the suspension homomorphism\footnote{The suspension $\pi_i(S^n) \rightarrow
\pi_{i+l}(S^{n+l})$ is an isomorphism if $i < 2n -1$ \cite[Corollary~4.24]{Hatcher:Algebraic-topology-2002}. For us this translates into
\begin{gather*}
s + (r-2)s_r -d(r-1) -1 < 2 (s_r (r-2)-1)-1, \ie  \\
\intertext{\ie} 
(s_1 - s_r) + \cdots + (s_{r-2} -s_r) + s_{r-1} + 2 < d(r-1),
\end{gather*}
which is trivially true if $m \leq d-3$.
}, hence $\beta = 0$. Therefore, using the third property in Lemma~\ref{app_suspended_maps}, we have reduced the problem to that of removing the $(r-1)$-intersection set between 
\[
\sigma_1 \cap \sigma_r, \ldots , \sigma_{r-1} \cap \sigma_r \subseteq \sigma_r,
\]
which are $(r-1)$ balls embedded in $\sigma_r$ in the metastable range\footnote{We must have for $i=1, \ldots,r-1$,
\begin{align*}
(r-1)s_r &\geq r(s_i + s_r -d ) + 3,\\
\intertext{\ie} r d &\geq r s_i + s_r + 3,
\end{align*}
and this is implied by
\[
rd \geq (r+1)m + 3 \geq r s_i + s_r + 3.
\]
} for $r-1$.

Thus, we are in position to work inductively: since $\sigma_r$ unknots in $B^d$, we have $B^d = \sigma_r * S^{d-s_r-1}$, so proper ambient isotopies of $\sigma_r$ can be extended to $B^d$.

The beginning of the induction (for $r=3$) reduces to the classical case of two balls intersecting inside a third ball, and is solved in Weber~\cite[ Prop.~1 \&~2]{Weber:Plongements-de-polyhedres-dans-le-domaine-metastable-1967}. 
\end{proof}

We are left with proving Lemma~\ref{lem_commuting_square}, this is what the rest of this section is devoted to. Before starting the proof (that will be split into three Lemmas in Section~\ref{sec_construction_square}), we introduce another kind of configuration space that will be useful for us during that proof.


\subsection{Deleted Joins} \label{app_deleted_joins}
Let $K$  be a simplicial complex. We define the \define{$k$-fold $k$-wise \emph{topological} deleted join} of $K$
\[
K^{*r} \setminus \delta_r^* K: = K * \cdots * K \setminus \left\{ \left. \frac 1r x + \cdots + \frac 1r x \; \right| \; x \in K\right\},
\]
and the \define{$k$-fold $k$-wise \emph{simplicial} deleted join} of $K$
\[
\delprodthin{(K)}r := \{ \tau_1 * \cdots * \tau_r \; | \; \tau_i \in K \text{ and } \tau_1 \cap \cdots \cap \tau_r =  \emptyset\}.
\]
Both spaces $K^{*r} \setminus \delta_r K$ and $\delprodthin{(K)}r $ have a natural $\sym_r$-action by permutation of the coordinates.

\begin{lemma}\label{lem_top_retracts_simpl}
$K^{*r} \setminus \delta_r^* K$ can be $\sym_r$-equivariantly retracted onto $\delprodthin{(K)}r $.
\end{lemma}
\begin{proof} Our proof is modelled on the deleted \textit{product} case \cite[Lemma~10.1]{Hu}.

\noindent\textbf{Warm up.}
We first show the main trick on a very simple case. \Ie assuming $\Delta$ is the simplex on two vertices $\{ x,y\}$, we construct an homeomorphism
\[
\Delta * \Delta \iso \delprodthin{(\Delta)}2 * \delta_2^*(\Delta), \quad \text{(see Figure~\ref{fig_del_join}).}
\]
Once we have this homeomorphism the conclusion is immediate.

First, we name the four vertices of  $\Delta * \Delta$ as  $\{ x,y,x',y'\}$ (with $\{x,y\} \in \Delta * \emptyset$ and $\{ x',y'\} \in \emptyset * \Delta$). Then every point of $\Delta * \Delta$ is represented as 
\[
x= a x + by + a' x' + b' y' \quad \text{with $a,a',b,b' \in [0,1]$},
\]

\begin{figure}[h!]
\begin{center}
\includegraphics[scale=1]{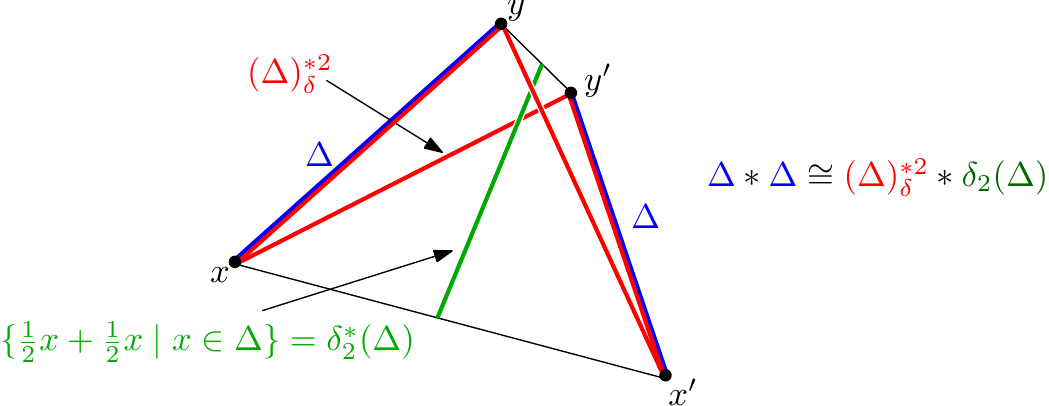}
\caption{The $k$-fold $k$-wise topological deleted join can be retracted to the $k$-fold $k$-wise simplicial deleted join.}
\label{fig_del_join}
\end{center}
\end{figure}

Assuming that $a  \geq a'$, $b \geq b'$ and that $a'$ or $b'$ is non-zero, we rewrite $x$ as
\begin{align*}
x &= (a-a' + b-b')\underbrace{\left(\frac{a-a'}{a-a' + b-b'} x + \frac{b-b'}{a-a' + b-b'}y \right)}_{\in\delprodthin{(\Delta)}2} + \\
& \quad \; (2 a' + 2 b' ) \underbrace{ \left(\frac{a'}{2a' + 2b'} (x+ x') + \frac{b'}{2a' + 2b'}(y+y')\right) }_{\in \delta_2^*(\Delta)}.
\end{align*}
The other possible orders on $a,b,a',b'$ can be worked on in a similar way, and will correspond to other faces of $\delprodthin{(\Delta)}2$.

\noindent\textbf{The general case.} Let $K$ be a simplicial complex. We can write any simplex of $K^{*r}$ as 
\[
\underbrace{(\Delta^1 * \omega^1)}_{\in K} * \cdots *\underbrace{( \Delta^r * \omega^r)}_{\in K} \quad \text{for some simplices $\Delta^i,\omega^i \in K$},
\]
with the condition
\[
\Delta^1 = \cdots = \Delta^r \quad \text{and} \quad \omega^1 \cap \cdots \cap \omega^r = \emptyset.
\] 
Our goal is to build an homeomorphism
\[
(\Delta^1 * \omega^1) * \cdots *( \Delta^r * \omega^r)
\iso
\delprodthin{(\Delta)}r * \delta_r^*(\Delta) * (\omega^1 * \cdots * \omega^r).
\]
where $\Delta$ is any of the $\Delta^i$. Once we have this homeomorphism the conclusion is immediate.

Let us name $p^i_j$ the vertices spanning $\Delta^i$, and $q^i_j$ the vertices spanning $\omega_i$. Then, any $x \in (\Delta^1 * \omega^1) * \cdots *( \Delta^r * \omega^r)$ can be written as
\[
x = \sum_{i,j} p^i_j(x) p^i_j + \sum_{i,j} q^i_j(x) q^i_j, \quad \text{with $p^i_j(x), q^i_j(x) \geq 0$ and } \sum_{i,j} p^i_j(x) + \sum_{i_j} q^i_j(x) = 1.
\]
We assume that at least one of the $q^i_j(x)$ is non-zero (otherwise nothing has to be done). Then, we write $x$ as 
\begin{align*}
x &= \sum p^i_j(x) \left( \frac 1{\sum p^i_j(x)} \sum p^i_j(x) p^i_j \right) + \sum q^i_j(x) \left( \frac 1{\sum q^i_j(x)} \sum q^i_j(x) q^i_j \right)
\end{align*}
The first term lies in $\Delta^1 * \cdots * \Delta^r$,  and the second in $\omega^1 * \cdots * \omega^r$. To further decompose the first term, we 
name $p_j$ the minimum of $\{ p^1_j(x), \ldots, p^r_j(x)\}$, then 
\begin{align*}
\sum_{i,j} p^i_j(x) p^i_j  &=  \sum_{i,j} (p^i_j(x) - p_j)p^i_j+ \sum_i (p_1 p_1^i + \cdots + p_r p^i_r)
\end{align*}
Hence, we can write $\sum_{i,j} p^i_j(x) p^i_j$ as a point in the join of $\delprodthin{(\Delta)}r$ and $\delta_r^*\Delta$.
\end{proof}


\subsection{Proof of Lemma~\ref{lem_commuting_square}}
\label{sec_construction_square}

We split the proof of Lemma~\ref{lem_commuting_square} in three steps.

\begin{lemma}[A First square]
Let $\sigma_1,...,\sigma_r$ be balls properly mapped to $B^d$ by $f \colon \sigma_1 \sqcup \cdots \sqcup \sigma_r \to B^d$, with the dimensional restrictions of Lemma~\ref{lem_whitney_trick}. 

Then the diagram
\begin{equation} \label{eq_diag_I}
\begin{gathered}
\xymatrix{
\boundary(\sigma_1 \times \cdots \times \sigma_r) \ar[d] \ar[r]^-{\iso} &
 \boundary \sigma_1  * \cdots * \boundary \sigma_r \ar[d]    \\
\boundary(B \times \cdots \times B)  \setminus \delta_r(B)   \ar[r]^-{\iso} &
 \boundary B * \cdots * \boundary B \setminus \delta_r(B) 
}
\end{gathered}
\end{equation}
commutes up to homotopy.

The map on the left is defined as before\footnote{It is easy to see that $\boundary(\sigma_1 \times \cdots \times \sigma_r) $ maps into 
\[
\boundary(B \times \cdots \times B) \setminus \delta_r(B)  \subseteq B \times \cdots \times B \setminus \delta_r(B) 
\]
since the $\sigma_i$ are \textit{properly} mapped in $B^d$. Also, if $B$ is represented as a cube $I^d = [-1,1]^d$, then 
\[
B \times \cdots \times B \setminus \{ (0 , \ldots ,0 )\}
\]
can be retracted onto its boundary, and this defines a retraction from 
\[
B \times \cdots \times B \setminus \delta_r(B)  \quad \text{to} \quad \boundary(B \times \cdots \times B) \setminus \delta_r(B) .
\]
}. The map on the right maps
\[
\emptyset * \cdots * \boundary \sigma_i * \cdots * \emptyset \rightarrow \emptyset * \cdots * \boundary B * \cdots * \emptyset
\]
and extends linearly. The two horizontal homeomorphisms are obtained in the following way: we represent $B^d$ as $I^d = [-1,1]^d$, then $\boundary B * \cdots * \boundary B$ can be formed inside of the cube $B \times \cdots \times B$, \ie
\[
\boundary B * \cdots * \boundary B \subseteq B \times \cdots \times B \quad \text{(Figure~\ref{fig_join_in_cube})}
\]
and by radial projection from the center of the cube, we get that $\boundary (B \times \cdots \times B)$ is homeomorphic with $\boundary B * \cdots * \boundary B$. This defines the bottom horizontal arrow, and the same construction work with the top horizontal arrow.

\begin{figure}
\begin{center}
\includegraphics[scale=1]{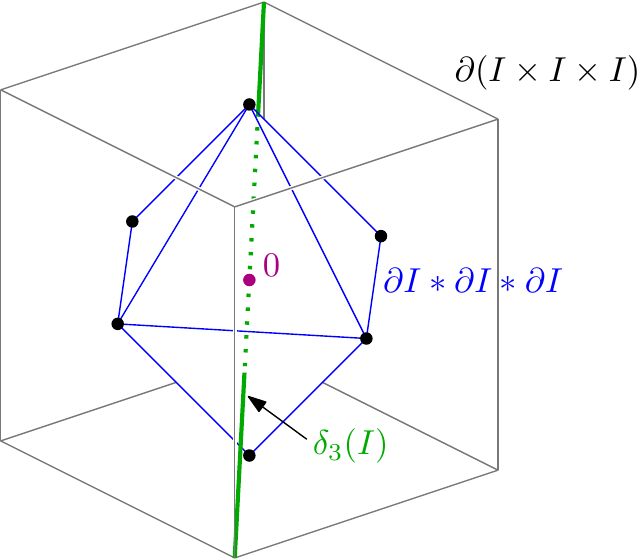}
\caption{For $r=3$ and $d=1$: the cube $I\times I \times I$ contains the join $\boundary I * \boundary I * \boundary I$.}
\label{fig_join_in_cube}
\end{center}
\end{figure}
\end{lemma}

\begin{proof}
The top left-to-right arrow is defined as (where $|.|$ is the infinity-norm)
\[
(x_1, \ldots, x_r ) \mapsto \frac{|x_1|}{\sum |x_i|} \underbrace{ \frac {x_1} {|x_1| }}_{\in \boundary \sigma_1}
\oplus \cdots  \oplus
\frac{|x_r|}{\sum |x_i|} \underbrace{  \frac {x_r} {|x_r| } }_{\in \boundary \sigma_r} \subset \boundary \sigma_1 * \cdots * \boundary \sigma_r
\]
with the convention that, if $|x_i| =0$, then $ \frac {x_i} {|x_i|}$ is undefined (but since its coefficient in the join is $0$, this is not a problem). The inverse application divides a point $p \in \boundary \sigma_1 * \cdots * \boundary \sigma_r \subset \sigma_1 \times \cdots \times \sigma_r$ by its $|.|$-norm (as a point in $\sigma_1 \times \cdots \times \sigma_r$) to project the point on the boundary $\boundary(\sigma_1 \times \cdots \times \sigma_r)$.

Starting from the top-left corner of Diagram~\eqref{eq_diag_I}, we follow the directions: right, down, left. We obtain a map $\boundary(\sigma_1 \times \cdots \times \sigma_r) \to \boundary(B \times \cdots \times B)  \setminus \delta_r(B)$ defined as
\begin{equation} \label{eq_diag_I_round_trip}
(x_1 , \ldots ,x_r) \mapsto 
\frac{ \big( \frac{|x_1|}{\sum |x_i|}  f\frac {x_1} {|x_1| }, \cdots , \frac{|x_r|}{\sum |x_i|} f\frac {x_r} {|x_r| } \big)  }
{ \big| \big( \frac{|x_1|}{\sum |x_i|}  f\frac {x_1} {|x_1| }, \cdots , \frac{|x_r|}{\sum |x_i|} f\frac {x_r} {|x_r| } \big) \big|}.
\end{equation}

To conclude, we must show that \eqref{eq_diag_I_round_trip} is homotopic to $(x_1 , \ldots ,x_r) \mapsto (fx_1,\ldots,fx_r)$.

Let us assume, without loss of generality, that $x_1 \in \boundary \sigma_1$. Then, $|x_1| = 1$, hence $ \big|  \frac{|x_1|}{\sum |x_i|}  f\frac {x_1} {|x_1| }  \big| = \frac{|fx_1|}{\sum |x_i|} = \frac{1}{\sum |x_i|}$. Therefore, the denominator in \eqref{eq_diag_I_round_trip} must be $\frac{1}{\sum |x_i|}$, and so \eqref{eq_diag_I_round_trip} becomes
\begin{equation}\label{eq_diag_I_round_trip_simpler}
\Big(
|x_1|f\frac{x_1}{|x_1|}, ... ,|x_r|f\frac{x_r}{|x_r|}
\Big)
\end{equation}
which is homotopic to  $(x_1 , \ldots ,x_r) \mapsto (fx_1,\ldots,fx_r)$ by a straight-line homotopy. Indeed, by contradiction, let us assume that for a given $(x_1,...,x_r)\in \boundary(\sigma_1\times \cdots\times \sigma_r)$ and a given $t\in(0,1)$, the straight-line homotopy intersects the diagonal $\delta_r(B)$. Without loss of generality, $x_1 \in \boundary \sigma_1$. But then, we must have $fx_2,...,fx_r \in \boundary B$, which implies that $x_2\in \boundary\sigma_2$, ..., $x_r\in\boundary \sigma_r$. But \eqref{eq_diag_I_round_trip_simpler} is the identify on such an $r$-tuple $(x_1,...,x_r)$, so it cannot intersect the diagonal.
\end{proof}

\begin{remark}
\begin{itemize}
\item 
We define $L:= {d(r-1) - s_r(r-2)}$ to shorten the exponent in $\Sigma^{d(r-1) - s_r(r-2)}$.
\item 
Recall that for $g$ a suspended map (Definition~\ref{def_suspended_map}), we define $\tau_i := g|_{\sigma_i}^{-1}\sigma_r$.
\end{itemize}
\end{remark}

\begin{lemma}[Second square.] 
Let $\sigma_1 \sqcup \, \cdots \, \sqcup \sigma_r \to B^d$ be a suspended map. Then the diagram
\begin{equation} \label{eq_diag_top}
\begin{gathered}
\xymatrix{
\boundary\sigma_1 * \cdots * \boundary \sigma_r \ar[d]  \ar[r]^-{\iso}  &
\Sigma^L( \boundary\tau_1  * \cdots * \boundary\tau_{r-1} )\ar[d]   \\
\boundary B * \cdots * \boundary B  \setminus \delta_r(B)    &
\Sigma^L ( \boundary \sigma_r * \cdots * \boundary \sigma_r  \setminus \delta_{r-1}(\sigma_r))  \ar@{_{(}->}[l]
}
\end{gathered}
\end{equation}
commutes. The map on the left is the obvious one, and the map on the right is the $(d(r-1) -s_r(r-2))$-suspension of the map 
\[
\boundary\tau_1  * \cdots * \boundary\tau_{r-1} \rightarrow  \boundary \sigma_r * \cdots * \boundary \sigma_r  \setminus \delta_{r-1}(\sigma_r).
\]
The horizontal maps are obtained as rearrangements using 
\[
B^d = \sigma_r * S^{d -s_r-1} \quad \text{and} \quad \sigma_i = \tau_i * S^{d -s_r-1}, \text{ for $i \neq r$.}
\]
More precisely, the top-horizontal homeomorphism is obtained in the following way
\begin{align*}
\boundary\sigma_1 * \cdots * \boundary \sigma_r  &=
\underbrace{\left(\boundary \tau_1 * S^{d-s_r-1}\right)}_{\boundary \sigma_1} * \cdots *
\underbrace{\left(\boundary \tau_{r-1} * S^{d-s_r-1}\right)}_{\boundary \sigma_r} * 
\boundary \sigma_r  \\
& \iso
\underbrace{(\boundary \sigma_r * S^{(r-1) (d-s_r)-1})}_{S^{(r-1)d -s_r(r-2)-1}} * \boundary \tau_1 * \cdots * \boundary \tau_{r-1},
\end{align*}
and this last expression is the suspension applied $((r-1)d -s_r(r-2))$-times on $\boundary \tau_1 * \cdots * \boundary \tau_{r-1}$. 

The bottom horizontal inclusion is derived, in a very similar fashion, as
\begin{multline} \label{eq_def_suspension_in_B}
S^{(r-1)d -s_r(r-2)-1}*  ( \boundary \sigma_r * \cdots * \boundary \sigma_r )  \\
\begin{aligned}
& \iso \left( \boundary \sigma_r * S^{d-s_r-1} \right) * \cdots * \left(\boundary \sigma_r * S^{d-s_r-1}\right) * (S^{s_r-1} * \emptyset )\\
&\iso \boundary B * \cdots * \boundary B * (\boundary \sigma_r * \emptyset) \\
& \subseteq  \boundary B * \cdots * \boundary B * \boundary B.
\end{aligned}
\end{multline}
\end{lemma}

\begin{proof}
It follows easily that \eqref{eq_diag_top} commutes. We show in the next step that the bottom horizontal inclusion is an homotopy equivalence, using Diagram~\eqref{eq_diag_top_to_simpl}.
\end{proof}

\begin{lemma}[Third square] \label{app_third_square}
Let $\sigma_1 \sqcup \, \cdots \, \sqcup \sigma_r \to B^d$ be a suspended map. Then the diagram
\begin{equation} \label{eq_diag_top_to_simpl}
\begin{gathered}
\xymatrix{
\boundary B * \cdots * \boundary B  \setminus \delta_r(B)    &
\Sigma^L ( \boundary \sigma_r * \cdots * \boundary \sigma_r  \setminus \delta_{r-1}(\sigma_r))  \ar@{_{(}->}[l] \\
\delprodthin{(\boundary B)} r   \ar@{_{(}->}^{\homeq}[u]&
\Sigma^L \delprodthin{ (\boundary \sigma_r) }{r-1}   \ar@{_{(}->}_{\homeq}[l] \ar@{_{(}->}^{\homeq}[u]
}
\end{gathered}
\end{equation}
commutes, and the three arrows with the symbol `$\homeq$' are homotopy equivalences.
\end{lemma}
\begin{proof}
Here,  $\delprodthin{(-)}k$ is the \emph{$k$-fold $k$-wise (simplicial) deleted join}. The definition is given in Section~\ref{app_deleted_joins}, where we also prove that both left and right vertical arrows are homotopy equivalences (Lemma~\ref{lem_top_retracts_simpl}).
Hence, we are left with the bottom-horizontal map.

We are going to use the two following facts that are easy to check: 
\begin{enumerate}
\item 
for any simplicial complexes $L_1, \ldots, L_n$,
$\delprodthin{(L_1 * \cdots * L_n)} k \iso \delprodthin {(L_1)}k * \cdots* \delprodthin {(L_n)}k$ , 
\item
$\delprodthin{(\boundary I)}k$ collapses simplicially onto $\boundary I * \cdots *\boundary I * \emptyset * \boundary I * \cdots * \boundary I$ (i.e., the $k$-join of $\boundary I$ where one of the factor is replaced by $\emptyset$).
\end{enumerate}

Therefore, if we represent $\sigma_r$ as $I^{s_r}$, we have
\begin{align*}
\delprodthin{(\boundary \sigma_r)}{r-1} &\iso \delprodthin{(\boundary I * \cdots * \boundary I)}{r-1} 
= \delprodthin{(\boundary I)}{r-1} * \cdots * \delprodthin{(\boundary I)}{r-1}.
\end{align*}
We collapse each of the $ \delprodthin{(\boundary I)}{r-1} $ to $ \boundary I * \cdots * \boundary I*\emptyset$ , hence $\delprodthin{(\boundary \sigma_r)}{r-1}$ collapses to
$
\boundary \sigma_r * \cdots * \boundary \sigma_r * \emptyset.
$
The suspension of this space in $\delprodthin{(\boundary B)} r $ is, by equation~\eqref{eq_def_suspension_in_B},
\[
\left( \boundary \sigma_r * S^{d-s_r-1} \right) * \cdots * \left(\boundary \sigma_r * S^{d-s_r-1}\right) *
 \left( \emptyset * S^{d-s_r-1} \right) 
* (S^{s_r-1} * \emptyset ),
\]
We can collapse $\delprodthin{(\boundary B)} r $ onto this last space. Indeed,
\begin{align*}
\delprodthin{(\boundary B)} r &= \delprodthin{(\boundary \sigma_r *S^{d-s_r-1} )} r
=  \delprodthin{(\boundary \sigma_r )}r * \delprodthin{(S^{d-s_r-1} )} r.
\end{align*}
The first term $\delprodthin{(\boundary \sigma_r )}r$ can be factors into terms $\delprodthin{(\boundary I)}r$, that we all collapse onto $\boundary I * \cdots * \boundary I * \emptyset * \boundary I$. 
For the second  term $ \delprodthin{(S^{d-s_r-1} )} r$, we collapse onto $\boundary I * \cdots * \boundary I * \emptyset $. 

Then, since both $\delprodthin{(\boundary B)} r $ and $\Sigma^L \delprodthin{(\boundary \sigma_r)}{r-1}$ can be collapsed onto the same sub-sphere, it follows that the bottom inclusion in \eqref{eq_diag_top_to_simpl} is an homotopy equivalence.
\end{proof}

\begin{proof}[Proof of Lemma~\ref{lem_commuting_square}]
Combining the all the previous Lemmas in this section, we get the following commuting diagram 
\begin{equation*}
\begin{gathered}
\xymatrix{
\boundary(\sigma_1 \times \cdots \times \sigma_r) \ar[d]  
 \ar[r]^-{\homeq}
 &
 \Sigma^L( \boundary \tau_1  * \cdots * \boundary\tau_{r-1} )\ar[d]  \\
\boundary(B \times \cdots \times B)  \setminus \delta_r(B)   \ar[r]^-{\homeq} &
 \Sigma^L ( \boundary \sigma_r * \cdots * \boundary \sigma_r  \setminus \delta_{r-1}(\sigma_r))
}
\end{gathered}
\end{equation*}
Reusing the first square~\eqref{eq_diag_I} on
\[
\boundary \tau_1 , \cdots, \boundary\tau_{r-1},
\]
we form
\begin{equation*}
\begin{gathered}
\xymatrix{
\Sigma^L( \boundary\tau_1  * \cdots * \boundary\tau_{r-1} )\ar[d]  \ar[r]^-{\homeq} 
&
\Sigma^L ( \boundary\tau_1 \times \cdots \times  \boundary\tau_{r-1}) \ar[d]
\\
\Sigma^L ( \boundary \sigma_r * \cdots * \boundary \sigma_r  \setminus \delta_{r-1}(\sigma_r)) \ar[r]^-{\homeq}
&
\Sigma^L ( \boundary(\sigma_r \times \cdots \times   \sigma_r  )  \setminus \delta_{r-1}(\sigma_r) )
}
\end{gathered}
\end{equation*}

Combining the last two diagrams, we get the diagram~\eqref{eq_diam_simple}, as wanted.
\end{proof}

%% file: proof-of-the-second-lemma-complete.tex
\subsection{The complete proof of Lemma~\ref{lem_whitney_trick}}

\begin{proof}[Proof of Lemma~\ref{lem_whitney_trick}]

\noindent\textbf{First Part.} We apply Proposition~\ref{prop_surg} to make each of the $\sigma_i \cap \sigma_r$ 
$
(s -d(r-1))\text{-connected}.
$ 
Then there exists, by Theorem~\ref{thm_Zeeman}, for each $i=1,\ldots, r-1$, a collapsible subspace $C_i$ of $\sigma_i \cap \sigma_r$ of dimension at most $s -d(r-1) +1$ such that
$
C_i \supseteq \sigma_1 \cap \cdots \cap \sigma_r.
$ 

In $\sigma_r$ there exists a collapsible space $C$ of dimension at most $s -d(r-1) +2$ containing $C_1 , \ldots, C_r$. Furthermore, by general position\footnote{We must have
\begin{align*}
(s_1 + \cdots + s_r -d(r-1) + 2) + (s_1 + s_r-d) -s_r &<0\\
\intertext{\ie} s_1 + \cdots + s_r + s_1 + 2 &< rd,
\end{align*}
which is implied by $(r+1)m + 2 < rd$.
}, $C$ intersects $\sigma_i \cap \sigma_r$ only on $C_i$. We take a regular neighbourhood $N $ of $C$ in $B^d$. By construction, $N$ intersects $\sigma_i \cap \sigma_r$ in a regular neighbourhood of $C_i$, which must be a ball ($C_i$ is disjoint for the other $\sigma_j$ for $j \neq i,r$ by general position)\footnote{This follows, once again by the metastable range hypothesis. Indeed,
\begin{align*}
(s_1 + \cdots + s_r -d(r-1) + 1) + (s_i + s_j + s_r-2d) -(s_i + s_r-d) &<0\\
\intertext{\ie} s_1 + \cdots + s_r + s_j + 1 &< rd,
\end{align*}
which is implied by $(r+1)m + 1 < rd$.
}. Hence, ``retracting'' from $B^d$ to the ball $N$ (as we did in the proof of Theorem~\ref{thm_Hae_Web_an}, equation~\ref{eq_retraction_sigma}), we are reduced to the situation of Proposition~\ref{prop_case_balls}, which we can then directly apply.

\paragraph{Second Part.} 
For $r=2$, the result already appeared in Weber~\cite[Proposition~3 \& proof of Lemma~1]{Weber:Plongements-de-polyhedres-dans-le-domaine-metastable-1967}.  \Ie if $\sigma^s$ and $\tau^t$ are two balls properly contained in $B^d$ in the metastable range ($m\geq s,t$ with $d \geq \frac 32 m + 3$) and without intersection. Then  for every $\alpha \in \pi_{s+t}(S^{d-1})$ there exists a proper isotopy $J_t$ of $B$ such that $J_1 \sigma \cap \tau = \emptyset$, and the homotopy class defined by
\[
\boundary (I \times \sigma \times \tau) \xrightarrow{J_t  \incl_{\sigma} \times \incl_{\tau}} B^d \times B^d \setminus \delta_2 B^d
\]
represents $\alpha$ (after identifications).

Hence, we can work inductively: we assume that the part~2 of the Lemma is already true for $(r-1)$ balls, and we show how construct the isotopy $J_t : B^d \rightarrow B^d$ for $r$ balls.

Let $\sigma_1,..., \sigma_r$ be the $r$ balls properly contained in $B^d$ as in the hypothesis of part 2 of the Lemma. In particular,
\[
\sigma_1 \cap \cdots \cap \sigma_r = \emptyset \quad \mbox{and} \quad \sigma_2 \cap \cdots \cap \sigma_r \neq \emptyset,
\]
and we can assume that $\sigma_r$ is unknotted in $B^d$, i.e., $B^d = \sigma_r * S^{d-s_r-1}$.

\begin{claim}
We can assume that for $i=1,...,(r-1),$ $\sigma_i\cap \sigma_r$ are balls properly contained inside $\sigma_r$. Furthermore, we can assume that $\sigma_2 \cap \cdots \cap \sigma_r$ is also a ball properly contained in $\sigma_r$.
\end{claim}
\begin{proof}
Let us pick $x \in \sigma_1 \cap \sigma_r$ and $y \in \sigma_2\cap \cdots\cap \sigma_r$, that we join by a path $\lambda \subseteq \sigma_r$ in general position. We take a regular neighborhood of $\lambda$ in $B^d$, and restrict ourselves to this neighborood. 
\end{proof}

By the induction hypothesis applied on 
\[
\sigma_1 \cap \sigma_r, ..., \sigma_{r-1} \cap \sigma_r \subseteq \sigma_r,
\]
for every homotopy class $\alpha \in \pi_{s + (r-2) s_r - (r-1)d+1}(S^{(r-2) s_r -1})$, there exists an isotopy $J_t$ of $\sigma_r$ such that $J_t$ applied to the ball $\sigma_1 \cap \sigma_r \subseteq \sigma_r$ represents $\alpha$. 

The isotopy $J_t$ can be extended to $B^d$ (we still denote it by $J_t$), hence this isotopy applied to the ball $\sigma_1 \subseteq B^d$ represents an homotopy class $\beta \in \pi_{s} (S^{d(r-1)-1})$. We are done if we can show that $\beta$ is a suspension of $\alpha$ (we are in the stable range of the suspension isomorphism).

The problem is similar to Lemma~\ref{lem_commuting_square}. 
Indeed we have $r$ balls 
\[
J(\sigma_1\times [-1,1]), \sigma_2, ..., \sigma_r
\]
that are mapped into $B^d$, and we would like to form a diagram as in \eqref{eq_diam_simple} with the ball `$\sigma_1 \times [-1,1]$' instead of $\sigma_1$.

Hence, to conclude, we only need to prove a version of the Suspended Map Lemma~\ref{app_suspended_maps} for our present situation.

Note that $\sigma_1 \times [-1,1]$ is not embedded inside $B^d$, and is not even properly mapped (the boundary is not mapped to the boundary).

\begin{claim}
Let $\t \sigma_1$ be a $(s_1 +1)$-ball mapped into $B^d$ with its boundary mapped as follows: $\boundary \t \sigma_1$ is decomposed into two balls $B_1$ and $B_2$ such that $B_1$ is mapped onto $J_0 \sigma_1$ and $B_2$ is mapped onto $J_1 \sigma_1$ (and $\boundary B_1 = \boundary B_2$ mapped onto $J_0 \boundary \sigma_1 = J_1\boundary \sigma_1$), then 
\[
J(\sigma_1 \times[-1,1]) \times \sigma_2 \times \cdots \times \sigma_r \quad \mbox{and} \quad \t \sigma_1 \times \sigma_2 \times \cdots \times \sigma_r
\]
define the same element $\beta \in \pi_{s} (S^{d(r-1)-1})$.
\end{claim}
\begin{proof}
This is immediate by using a straight line homotopy between $J(\sigma_1 \times [-1,1])$ and $\t \sigma_1$
\end{proof}
So we are reduced with working with a $(s_1+1)$-ball $\t \sigma_1$ instead of $J(\sigma_1\times [-1,1])$, and the way that we `fill' this ball does not matter (only the boundary decides of the homotopy class $\beta$).

We can decompose $\boundary \t \sigma_1$ as two balls $B_1$ and $B_2$ both homeomorphic with $(\sigma_1 \cap \sigma_r) * S^{d-s_r-1}$, and with $B_1 \cap B_2 \homeq \boundary (\sigma_1 \cap \sigma_r) * S^{d-s-r}$.

\begin{claim}
We can assume that $B_1$ is mapped to $(\sigma_1 \cap \sigma_r) *S^{d-s_r-1}$ and that $B_2$ is mapped to $(J_1 \sigma_1 \cap \sigma_r) * S^{d-s_r-1}$.
\end{claim}
\begin{proof}
This follows by an argument identical to that of Lemma~\ref{app_suspended_maps} (we work with the two balls seperately).
\end{proof}

\begin{claim}
We can assume that $B_1 = (\sigma_1 \cap \sigma_r) * S^{d-s_r-1}$ and $B_2 = (J_1 \sigma_1 \cap \sigma_r) * S^{d-s_r-1}$ are mapped onto the boundary of $B^d$, and that $\t \sigma_1 = ( \t \sigma_1 \cap \sigma_r) * S^{d-s_r-1}$.
\end{claim}
\begin{proof}
Figure~\ref{fig_retraction_end_induction} illustrate the construction inside $\sigma_r$.
\begin{figure}
\begin{center}
\includegraphics[scale=1]{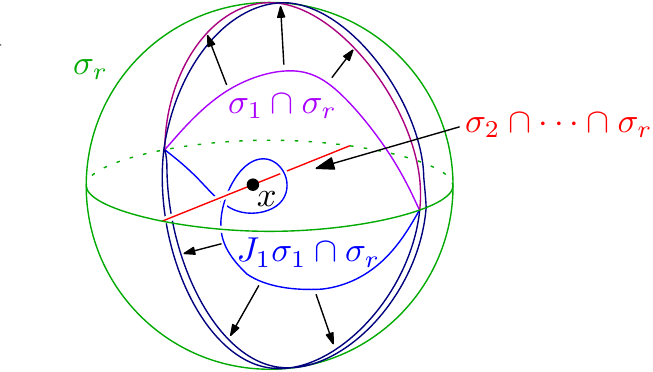}
\caption{Retraction of the sphere $(B_1 \cap \sigma_r) \cup (B_2\cap \sigma_r)$ on the boundary of $\sigma_r$.}
\label{fig_retraction_end_induction}
\end{center}
\end{figure}

We pick a point $x $ in the interior of the ball $ \sigma_2\cap \cdots\cap \sigma_r$. Since this ball unknots in $\sigma_r$, there exists a retraction $r_t$ of  $\sigma_r \setminus x$ to $\boundary \sigma_r$ such that $r_1^{-1}$ $\boundary(\sigma_2\cap \cdots\cap \sigma_r) = \sigma_2\cap \cdots\cap \sigma_r$.

Using that $B^d = \sigma_r * S^{d-s_r-1}$, we extend $r_t$ to $B_d$ (which now retract $B^d \setminus x$ to $\boundary B^d$). We can then use $r_t$ to conclude. (Note that $B_1$ and $B_2$ stop to be embedded, but this is not a problem for us). 
\end{proof}

We can now apply Lemma~\ref{lem_commuting_square} to the balls $\t\sigma_1, \sigma_2 ,...,\sigma_r$, and thus conclude. 
\end{proof}

%% file: BB.tex
\section{Block Bundles}
\label{chap_block_bundles}

In this Appendix, we review the notions of block bundles that are used in Section~\ref{sec_increasing_the_connectivity}. The theory of Block Bundles was developped by Rourke and Sanderson in a series of three papers \cite{BBI,BBII,BBIII} (see also \cite[p.~236]{Bryant:Piecewise-linear-topology-2002}, \cite{rourke1966block, 
buonchristiano1976geometric}, and Casson's chapter in \cite{ranicki1997hauptvermutung}). The goal of this theory is to \emph{establish for the piecewise-linear category a tool analogous to the vector bundle in the differential categroy}. \cite[p.~1]{BBI}.

\begin{definition}[PL cell complex]
A \define{complex} $K$ is a finite collection of cells which cover a polyhedron $X$ satisfying: For all $\sigma, \tau \in K$ distinct
\begin{enumerate}[(1)]
\item$\boundary \sigma$ and $\sigma\cap \tau$ are unions of cells of $K$,
\item $\interior \sigma\cap \interior \tau = \emptyset$.
\end{enumerate}
\end{definition}

\begin{remarks}
\begin{enumerate}[(a)]
\item We write $|K| := X$, when $K$ is a complex associate to a polyhedron $X$.
\item We often identified $\sigma\in K$ with the subcomplex of $K$ that it determines.
\end{enumerate}
\end{remarks}

\begin{definition}[Product of Complexes]
Let $K$ and $L$ be two complexes. The collection of cells 
\[
K \times L := \{ \sigma\times \tau \; | \; \sigma\in K \mbox{ and } \tau \in L\}
\]
is also a complex, whose geometric realization is homeomorphic to $|K| \times |L|$.
\end{definition}

\begin{definition}[Subdivision]
Let $K$ and $ K'$ be two cell complex structures of a polyhedron $X$. We say that $ K'$ is a \define{subdivision} of $K$ if 
\[
\mbox{for all }\sigma'\in K',\mbox{ there exists }\sigma\in K\mbox{ with }\sigma'\subseteq \sigma.
\]
\end{definition}

%
%

\begin{definition}[Block Bundle]
A \define{$q$-block bundle} $\xi^q$ is a pair $(E,K)$ of cell complexes such that 
\begin{enumerate}[(1)]
\item $|K| \subset |E|$
\item For each $n$-cell $\sigma\in K$, there exists an $(n+q)$-ball $\beta \in E$ such that
\[
(\beta,  \sigma ) \iso (I^{n+q},I^n).
\]
We call $\beta$ the \define{block} over $\sigma$.
\item $E$ is the union of the blocks $\beta$. 
\item The interiors of the blocks are disjoints.
\item Let $\beta,\gamma\in E$ be the block over $\sigma,\tau\in K$, respectively. Let $L= \sigma \cap \tau \subset K$, then $\beta \cap \gamma \subset E$ is the union of the blocks over cells of $L$.
\end{enumerate}
\end{definition}

\begin{remarks}
\begin{enumerate}[(a)]
\item For a $q$-block bundle $\xi^q = (E,K)$, we denote $E(\xi) := E$, and the whole bundle data is denoted by 
\[
\xi^q / K.
\]
Sometimes, we simply write $\xi / X$, where $X= |K|$.
\item The \define{trivial $q$-block bundle} over $K$ is denoted by $\epsilon^q /K$. For any $\sigma_i\in K$
\[
\beta_i := \sigma_i \times [-1,1]^q.
\]
\end{enumerate}
\end{remarks}

%

\begin{definition}[Isomorphism of Block Bundles]
Two block bundles $\xi,\eta /K$ are \define{isomorphic} (written $\xi \iso \eta$) if there exists a homeomorphism
\[
h : E(\xi) \to E(\eta),
\]
such that
\begin{enumerate}[(1)]
\item its restriction to $K$ is the identity ($h_K=1$), and
\item Let $\sigma \in K$ with $\beta \in E(\xi)$ and $\gamma\in E(\eta)$ the two blocks over $\sigma$, then
\[
h(\beta) = \gamma.
\]
\end{enumerate}
\end{definition}

\begin{definition}[Subdivision]
Let $\eta /K$ be a block bundle, and let $K'$ be a subdivision of $K$. A block bundle $\eta'/K'$ is a \define{subdivision} of $\eta/K$ if for all $\sigma \in K$ (with associated block $\beta$), such that 
\[
|\sigma| = |\interior \tau_1| \sqcup \cdots \sqcup |\interior \tau_n|
\]
where $\tau_1,\ldots, \tau_n \in K' $, then
\[
|\beta| = |\interior \mu_1| \sqcup \cdots \sqcup |\interior \mu_n|
\]
where $\mu_i$ is the block over $\tau_i$ in $\xi'/K'$. 
\end{definition}

\begin{definition}[Equivalence of Block Bundles]
Let $\xi/K$ and  $\eta/L$ be two block bundles such that $|K| = |L|$. We say that $\xi/K$ is \define{equivalent} to $\eta/L$ if there exists subdivisons $\xi'/K'$ and $\eta'/L'$ such that $\xi' \iso \eta'$.
\end{definition}

\begin{definition}
Let $K$ be a polyhedron.
\begin{enumerate}[(1)]
\item 
The set of isomorphism classes of $q$-block bundles over $K$ is denoted $I_q(K)$.
\item 
The set of equivalence classes of $q$-block bundles over $|K|$ is denoted $I_q(|K|)$.
\end{enumerate}
\end{definition}

\begin{theorem}[Thm~1.10 \cite{BBI}]
Let $K$ be a polyhedron. Then there is a bijection
\[
I_q(K) \to I_q(|K|),
\] 
defined by associating each block bundle over $K$ to its equivalence class.
\end{theorem}

We have the following analogue of the `tubular neighborhood Theorem' for block bundles:
\begin{theorem}[Sec.~4 of \mbox{\cite{BBI}}  or
Thm~3.27 in \mbox{\cite{Bryant:Piecewise-linear-topology-2002}}]
\label{thm_corr_regular_neighborhood_bb}
Suppose $(M,Q)$ is a $(m,q)$-manifold pair with $m-q \ge 3$. Every regular neighborhood of $Q$ in $M$ has the structure of an $(m-q)$-block bundle over $Q$.
\end{theorem}

\begin{definition}[Abstract Regular Neighborhood]
Suppose $M^m \subset Q^{q}$ is a proper submanifold of $Q$ (both compact). We say that $Q$ is an \define{abstract regular neighborhood} of $M$ is $Q$ collapses to $M$.
\end{definition}

\begin{theorem}[Cor.~4.6 in \cite{BBI}]
Let $M$ be a $m$-manifold, and $\mathfrak R_q(M)$ be the set of  homeomorphism ($\mathrm{mod}\; M$) classes of abstract regular neighborhoods of $M$ of dimension $m+q$.

There is a bijection
\[
I_q (M)  \to
\mathfrak R_q(M)
\] 
given by associating every block bundle to the corresponding abstract regular neighborhood, i.e., $\xi/M$ is sent to $E(\xi)$.
\end{theorem}

\begin{definition}[Restriction]
Let $L,K$ be two complexes such that $L \subset K$. Given a block bundle $\xi/K$, we define its \define{restriction} to $L$ (notation: $\xi|L$) which is the block bundle over $L$ obtained from $\xi/K$ by restricting to the block over cells of $L$.

I.e., given $\sigma\in L$, then the block over $\sigma$ in $\xi|L$ is the block over $\sigma$ in $\xi/K$.
\end{definition}
\begin{remark}
Let $u/X$ be an equivalence class of block bundle over a polyhedron $X$. Let $Y \subset X$ be a closed subpolyhedron of $X$. Then the restriction $u | Y$ is a well-defined equivalence  class of block bundles over $Y$ (see \cite[p.~8]{BBI}). 
\end{remark}

\begin{definition}[Cartesian Product]
Suppose $\xi/K$ and $\eta /L$ are two block bundles. Their \define{cartesian product} $(\xi \times \eta) / (K \times L)$ is the block bundle over $K\times L$ defined by
\begin{enumerate}[(1)]
\item 
$E(\xi \times \eta) := E(\xi) \times E(\eta)$.
\item 
Given $\sigma \in K$, $\tau \in L$ with $\beta \in E(K)$ the block over $\sigma$ and $\gamma \in E(L)$ the block over $\tau$. Then the block over $\sigma\times \tau \in K \times L$ is $\beta \times \gamma$.
\end{enumerate}
\end{definition}
\begin{remark}
As for the restriction, the cartesian product is also well-defined on \emph{equivalence classes} of block bundles.
\end{remark}

\begin{definition}[Whitney Sum]
Given two equivalences classes of block bundles $u/X$ and $v/X$. We define their \define{Whitney sum} 
\[
(u \oplus v)/X := (u \times v) | \Delta,
\]
where $\Delta = \{ (x,x) \in X \times X\}$ is identified with $X$ by the diagonal map.
\end{definition}

\begin{theorem}[Thm~2.7 \mbox{\cite{BBI}}]
For all $q \ge 0$, there exists a locally finite simplicial complex $B\t{PL}_q$ such that: For every polyhedron $X$ 
there is a bijection
\[
[X,B\t{PL}_q] \to I_q(X).
\]
\end{theorem}
\begin{remark}
The above theorem is the PL analogue of the classification of vector bundles over a smooth manifold $M$: equivalences classes of such vector bundles are in bijection with $[M,G_n(\R^{\infty})]$, where $G_n(\R^{\infty})$ is the infinite $n$-Grassmanian.
\end{remark}

\begin{theorem}[Stability of Block Bundles, Cor.~5.2--5.4 \mbox{\cite{BBII}}]
\label{thm_stability_of_bb}
Let $\xi^q/K^k$ be a $q$-block bundle over a $k$-complex $K$.
\begin{enumerate}[(a)]
\item 
There exists $\eta^k / K^k$ such that
\[
\xi^q \iso \eta^k \oplus \epsilon^{q-k}.
\]
\item 
There exists $\eta^k/K^k$ such that 
\[
\xi^q \oplus \eta^k  \iso \epsilon^{q+k}.
\]
\item 
For $\xi^{k+1}/K^k$ and $\eta^{k+1}/K^k$,
\[
 \xi^{k+1}\oplus \epsilon^t   \iso \eta^{k+1}\oplus \epsilon^t \quad \Rightarrow \quad \xi^{k+1}\iso \eta^{k+1}.
 \]
\end{enumerate}
\end{theorem}

\begin{definition}[Transversality]
\label{def_transversality}
Let $M,N\subset Q$ be compact proper submanifolds, and $\xi/M \subset Q$ a block bundle over $M$ in $Q$. (In particular, $E(\xi)$ is a regular neighborhood of $M\subset Q$.)
\begin{enumerate}[(a)]
\item 
We say that $N$ is \define{transverse to $M$ with respect to $\xi$} if there exists a subdivision $\xi'$ of $\xi$ such that 
\[
N \cap E(\xi) = E(\xi' | N \cap M). 
\]
\item
We say that $N$ is \define{locally} transverse to $M$ with respect to $\xi$ if this is true near $M$, i.e., there exists a subdivision $\xi'$ of $\xi$ and a regular neighborhood $U$ of $M$ with 
\[
N \cap \xi \cap U = E(\xi' | N\cap M) \cap U.
\]
\end{enumerate} 
\end{definition}

\begin{theorem}[Thm~1.1 in \mbox{\cite{BBII}}]
\label{thm_transversality}
Suppose $N,M \subset Q$ are proper submanifolds, and $\xi$ is a block bundle on $M \subset Q$. 
\begin{enumerate}[(a)]
\item 
There is an $\epsilon$-isotopy of $Q$ carrying $N$ locally transverse to $M$ with respect to $\xi$.
\item 
If $\boundary N$ is already locally transvered to $\boundary M$ with respect to $\xi|\boundary M$, then the isotopy may be taken $\mathrm{mod}\; \boundary Q$.
\end{enumerate}
\end{theorem}

\begin{definition}[Tangent Bundles]
Given a manifold $M$, the regular neighborhood of $\Delta_M \subset M \times M$ defines a block bundle. Its equivalence class, denoted $t(M)$, is called the \define{tangent bundle} of $M$.
\end{definition}

We have the following analogue of the decomposition of the tangent vector bundle over a submanifold in the smooth category (it  decomposes as the tangent bundle $\oplus$ the normal bundle):
\begin{theorem}[Prop.~5.5 and Cor.~5.6 in \mbox{\cite{BBII}}]
\label{thm_tangent_normal}
\begin{enumerate}[(a)]
\item 
Let $M$ be a manifold wiht $\xi /M$ a block bundle over $M$. (In particular, $E(\xi)$ is an abstract regular neighborhood of $M$.)  Then
\[
t(E(\xi)) | M \iso ( \xi \oplus t(M) ) / M. 
\]
\item 
Let $M \subset Q$ be a compact submanifold, and $u$ the class of any normal block bundle on $M$ in $Q$, then
\[
( t(M) \oplus u ) / M \iso t(Q) |M.
\]
\end{enumerate}
\end{theorem}

\begin{proposition}[Tangent bundle of a sphere is stably trivial]
\label{prop_s_k_stab_trivial}
Let $S^k$ be a $k$-sphere, then
\[
t(S^k) \oplus \epsilon^1 \iso \epsilon^{k+1}
\]
\end{proposition}
\begin{proof}
Let us consider $S^k \subset \R^{k+1}$. By the existence of bicollars, any normal bundle of $S^k \subset  \R^{k+1}$ is isomorphic to $\epsilon^1 /S^k$.

Then, by Theorem~\ref{thm_tangent_normal},
\[
\epsilon^1 \oplus t(S^k ) = t(\R^{k+1}) | S^k = \epsilon^{k+1}/S^k. 
\qedhere
\]
\end{proof}

\begin{proposition}[PL analogue of Thm~2 in \cite{Milnor}]
\label{prop_s_k_in_M_pi_manifold}
Let $S^k$ be a $k$-sphere embedded in an $m$-manifold $M^m$, with 
\[
2k + 1 \leq m.
\]
Assume that the tangent bundle of $M$ is \emph{stably trivial}, i.e., there exists $l\ge 0$ such that
\[
t(M) \oplus \epsilon^l \iso \epsilon^{m+l}.
\]
Then any normal bundle over $S^k$ in $M$ is trivial.
\end{proposition}
\begin{proof}
Let $\xi^{m-k}/S^k$ be a normal bundle of $S^k \subset M$. 

Let us form $M \times [-1,1]^l$ (i.e., $\epsilon^l / M$, note that we can always assume $l\ge 1$).

A normal bundle of $S^k \subset M \times [-1,1]^l$ is
\[
\xi^{m-k} \oplus \epsilon^{l}.
\]
By Theorem~\ref{thm_tangent_normal}  (for the first and second equalities),
\begin{align*}
( t(S^k) \oplus \xi^{m-k} \oplus \epsilon^{l} ) / S^k 
&=
t(M\times [-1,1]^l) | S^k
\\
&=
( t(M) \oplus \epsilon^l ) | S^k
\\
&= \epsilon^{m+l} | S^k.
\end{align*}

By Theorem~\ref{prop_s_k_stab_trivial} 
\[
( \xi^{m-k} \oplus t(S^k) \oplus \epsilon^l  ) /S^k 
= 
( \xi^{m-k} \oplus \epsilon^{k+l} ) / S^k. 
\]
Combining the last two equations (we do not write ``$/S^k$'' anymore)
\[
 \xi^{m-k} \oplus \epsilon^{k+l}  
=
\epsilon^{m+l} .
\]
By Theorem~\ref{thm_stability_of_bb}, there exists $\eta^k/S^k$ such that
\[
\xi^{m-k} = \eta^{k} \oplus \epsilon^{m-2k}
\]
So, by the last two equations,
\[
 \eta^{k} \oplus \epsilon^{m-2k + k+ l}  = \epsilon^{m+l}.
\]
Again by Theorem~\ref{thm_stability_of_bb}, this implies
\[
\eta^k \oplus \epsilon^1 = \epsilon^{k+1}
\]
So
\[
\xi^{m-k} = \eta^k \oplus \epsilon^{m-2k} = \epsilon^{m-k},
\]
since $m-2k\ge 1$.
\end{proof}

%% file: MAIN.bbl
\newcommand{\etalchar}[1]{$^{#1}$}
\providecommand{\noopsort}[1]{}
\begin{thebibliography}{{\v{C}}KM{\etalchar{+}}14b}

\bibitem[AMSW15]{Avvakumov:Eliminating-Higher-Multiplicity-Intersections-III.-2015}
S.~Avvakumov, I.~Mabillard, A.~Skopenkov, and U.~Wagner.
\newblock Eliminating higher-multiplicity intersections, {III}. {C}odimension
  2.
\newblock {\em Preprint \href{http://arxiv.org/abs/1511.03501}{\tt
  arXiv:1511.03501}}, 2015.

\bibitem[BB79]{Bajmoczy:On-a-common-generalization-of-Borsuks-and-Radons-1979}
E.~G. Bajm{{\'o}}czy and I.~B{{\'a}}r{{\'a}}ny.
\newblock On a common generalization of {B}orsuk's and {R}adon's theorem.
\newblock {\em Acta Math. Acad. Sci. Hungar.}, 34(3-4):347--350 (1980), 1979.

\bibitem[BBZ16]{Barany:Tverbergs-theorem-at-50:-extensions-and-counterexamples-2016}
Imre B{{\'a}}r{{\'a}}ny, Pavle V.~M. Blagojevi{{\'c}}, and G{{\"u}}nter~M.
  Ziegler.
\newblock Tverberg's theorem at 50: extensions and counterexamples.
\newblock {\em Notices Amer. Math. Soc.}, 63(7):732--739, 2016.

\bibitem[BDB07]{Blagojevic:Using-equivariant-obstruction-theory-in-combinatorial-geometry-2007}
Pavle V.~M. Blagojevi{{\'c}} and Aleksandra~S.
  Dimitrijevi{{\'c}}~Blagojevi{{\'c}}.
\newblock Using equivariant obstruction theory in combinatorial geometry.
\newblock {\em Topology Appl.}, 154(14):2635--2655, 2007.

\bibitem[BFL90]{Barany:On-the-number-of-halving-planes-1990}
Imre B\'ar\'any, Zolt\'an F\"uredi, and L\'aszlo Lov\'asz.
\newblock On the number of halving planes.
\newblock {\em Combinatorica}, 10(2):175--183, 1990.

\bibitem[BFZ14]{Blagojevic:Tverberg-plus-constraints-2014}
Pavle V.~M. Blagojevi{{\'c}}, Florian Frick, and G{{\"u}}nter~M. Ziegler.
\newblock Tverberg plus constraints.
\newblock {\em Bull. Lond. Math. Soc.}, 46(5):953--967, 2014.

\bibitem[BFZ15]{Blagojevic:Barycenters-of-Polytope-Skeleta-and-Counterexamples-2015}
Pavle V.~M. Blagojevi{\'c}, Florian Frick, and G{\"u}nter~M. Ziegler.
\newblock Barycenters of polytope skeleta and counterexamples to the
  topological tverberg conjecture, via constraints.
\newblock Preprint, \texttt{arXiv:1510.07984}, 2015.

\bibitem[BL92]{Barany:A-colored-version-of-Tverbergs-theorem-1992}
I.~B{{\'a}}r{{\'a}}ny and D.~G. Larman.
\newblock A colored version of {T}verberg's theorem.
\newblock {\em J. London Math. Soc. (2)}, 45(2):314--320, 1992.

\bibitem[BMZ15]{Blagojevic:Optimal-bounds-for-the-colored-Tverberg-problem-2009}
P.~V.~M. Blagojevi\'{c}, B.~Matschke, and G.~M. Ziegler.
\newblock Optimal bounds for the colored {T}verberg problem.
\newblock {\em J. Eur. Math. Soc.}, 17(4):739--754, 2015.

\bibitem[BRS76]{buonchristiano1976geometric}
Sandro Buonchristiano, Colin~Patrick Rourke, and Brian~Joseph Sanderson.
\newblock {\em A geometric approach to homology theory}, volume~18.
\newblock Cambridge University Press, 1976.

\bibitem[Bry02]{Bryant:Piecewise-linear-topology-2002}
John~L. Bryant.
\newblock Piecewise linear topology.
\newblock In {\em Handbook of geometric topology}, pages 219--259.
  North-Holland, Amsterdam, 2002.

\bibitem[BSS81]{Barany:On-a-topological-generalization-of-a-theorem-of-Tverberg-1981}
I.~B\'{a}r\'{a}ny, S.~B. Shlosman, and A.~Sz{\H{u}}cs.
\newblock On a topological generalization of a theorem of {T}verberg.
\newblock {\em J. London Math. Soc., II. Ser.}, 23:158--164, 1981.

\bibitem[BZ16]{Blagojevic:Beyond-the-Borsuk-Ulam-theorem:-The-topological-2016}
Pavle V.~M. Blagojevi{\'c} and G{\"u}nter~M. Ziegler.
\newblock Beyond the borsuk-ulam theorem: The topological tverberg story.
\newblock Preprint, \texttt{arXiv:1605.07321}, 2016.

\bibitem[CH34]{Hanani:UnplattbareKurven-1934}
Haim Chojnacki~(Hanani).
\newblock {\"Uber wesentlich unpl\"attbare Kurven im dreidimensionalen Raume.}
\newblock {\em Fund. Math.}, 23:135--142, 1934.

\bibitem[{\v{C}}KM{\etalchar{+}}14a]{Cadek:Computing-all-maps-into-a-sphere-2014}
Martin {\v{C}}adek, Marek Kr{\v{c}}{{\'a}}l, Ji{\v{r}}{\'{\i}} Matou{\v{s}}ek,
  Francis Sergeraert, Luk{{\'a}}{\v{s}} Vok{\v{r}}{\'{\i}}nek, and Uli Wagner.
\newblock Computing all maps into a sphere.
\newblock {\em J. ACM}, 61(3):Art. 17, 44, 2014.

\bibitem[{\v{C}}KM{\etalchar{+}}14b]{Cadek:Polynomial-time-computation-of-homotopy-groups-and-Postnikov-systems-2014}
Martin {\v{C}}adek, Marek Kr{\v{c}}{{\'a}}l, Ji{\v{r}}{\'{\i}} Matou{\v{s}}ek,
  Luk{{\'a}}{\v{s}} Vok{\v{r}}{\'{\i}}nek, and Uli Wagner.
\newblock Polynomial-time computation of homotopy groups and {P}ostnikov
  systems in fixed dimension.
\newblock {\em SIAM J. Comput.}, 43(5):1728--1780, 2014.

\bibitem[{\v
  C}KV13]{Cadek:Algorithmic-solvability-of-the-lifting-extension-problem-2013}
Martin {\v C}adek, Marek Kr{\v c}{\'a}l, and Luk{\'a}{\v s} Vok{\v r}{\'\i}nek.
\newblock Algorithmic solvability of the lifting-extension problem.
\newblock Preprint, \texttt{arXiv:1307.6444}, 2013.

\bibitem[FKT94]{Freedman:van-Kampens-embedding-obstruction-is-incomplete-for-2-complexes-in-bf-R4-1994}
Michael~H. Freedman, Vyacheslav~S. Krushkal, and Peter Teichner.
\newblock van {K}ampen's embedding obstruction is incomplete for
  {$2$}-complexes in {${\bf R}^4$}.
\newblock {\em Math. Res. Lett.}, 1(2):167--176, 1994.

\bibitem[Fri15]{Frick:Counterexamples-to-the-topological-Tverberg-conjecture-2015}
Florian Frick.
\newblock Counterexamples to the topological {T}verberg conjecture.
\newblock Preprint, \texttt{arXiv:1502.00947}, 2015.

\bibitem[Fri16]{Frick:Intersection-patterns-of-finite-sets-2016}
Florian Frick.
\newblock Intersection patterns of finite sets and of convex sets.
\newblock Preprint, \texttt{arXiv:1607.01003}, 2016.

\bibitem[FV16]{FilakovskyVokrinek}
Marek Filakovsk{\'{y}} and Luk{\'a}{\v{s}} Vok{\v{r}}{\'{\i}}nek.
\newblock Computing homotopy classes in diagrams.
\newblock In preparation, 2016.

\bibitem[Gro10]{gromov2010singularities}
Mikhail Gromov.
\newblock Singularities, expanders and topology of maps. part 2: From
  combinatorics to topology via algebraic isoperimetry.
\newblock {\em Geometric and Functional Analysis}, 20(2):416--526, 2010.

\bibitem[GS79]{Gruber:Problems-in-geometric-convexity-1979}
P.~M. Gruber and R.~Schneider.
\newblock Problems in geometric convexity.
\newblock In {\em Contributions to geometry ({P}roc. {G}eom. {S}ympos.,
  {S}iegen, 1978)}, pages 255--278. Birkh{\"a}user, Basel-Boston, Mass., 1979.

\bibitem[GS06]{GoncalvesSkopenkov:EmbeddingsHomologyEquivalentManifolds-2006}
Daciberg Gon{\c{c}}alves and Arkadiy Skopenkov.
\newblock Embeddings of homology equivalent manifolds with boundary.
\newblock {\em Topology Appl.}, 153(12):2026--2034, 2006.

\bibitem[Hae63]{Haefliger:Plongements-differentiables-dans-le-domaine-stable-1962}
Andr{{\'e}} Haefliger.
\newblock Plongements diff{\'e}rentiables dans le domaine stable.
\newblock {\em Comment. Math. Helv.}, 37:155--176, 1962/1963.

\bibitem[Hae66]{haefliger1966enlacements}
Andr{\'e} Haefliger.
\newblock Enlacements de sph{\`e}res en codimension sup{\'e}rieure {\`a} 2.
\newblock {\em Commentarii Mathematici Helvetici}, 41(1):51--72, 1966.

\bibitem[Hat02]{Hatcher:Algebraic-topology-2002}
Allen Hatcher.
\newblock {\em Algebraic Topology}.
\newblock Cambridge University Press, Cambridge, 2002.

\bibitem[Hu60]{Hu}
Sze-{T}sen Hu.
\newblock Isotopy invariants of topological spaces.
\newblock {\em Proc. Roy. Soc. London. Ser. A}, 255:331--366, 1960.

\bibitem[Hud66]{Hudson:Extending-piecewise-linear-isotopies-1966}
John F.~P. Hudson.
\newblock Extending piecewise-linear isotopies.
\newblock {\em Proc. London Math. Soc. (3)}, 16:651--668, 1966.

\bibitem[Hud69]{Hudson}
J.~F.~P. Hudson.
\newblock {\em Piecewise linear topology}.
\newblock University of Chicago Lecture Notes prepared with the assistance of
  J. L. Shaneson and J. Lees. W. A. Benjamin, Inc., New York-Amsterdam, 1969.

\bibitem[HZ64]{HudsonZeemanRN1}
J.~F.~P. Hudson and E.~C. Zeeman.
\newblock On regular neighbourhoods.
\newblock {\em Proc. London Math. Soc. (3)}, 14:719--745, 1964.

\bibitem[JV{\v{Z}}16]{jojic2016topology}
Du{\v{s}}ko Joji{\'c}, Sini{\v{s}}a~T Vre{\'c}ica, and Rade~T
  {\v{Z}}ivaljevi{\'c}.
\newblock Topology and combinatorics of 'unavoidable complexes'.
\newblock {\em Preprint \href{http://arxiv.org/abs/1603.08472}{\tt
  arXiv:1603.08472}}, 2016.

\bibitem[Kar08]{Karasev:Topological-methods-in-combinatorial-geometry-2008}
R.~N. Karas{\"e}v.
\newblock Topological methods in combinatorial geometry.
\newblock {\em Uspekhi Mat. Nauk}, 63(6(384)):39--90, 2008.

\bibitem[Mat03]{Matousek:BorsukUlam-2003}
Ji{\v{r}}{\'{\i}} Matou{\v{s}}ek.
\newblock {\em Using the {B}orsuk-{U}lam theorem}.
\newblock Springer-Verlag, Berlin, 2003.

\bibitem[Mil61]{Milnor}
J.~Milnor.
\newblock A procedure for killing homotopy groups of differentiable manifolds.
\newblock {\em Proc. Sympos. Pure Math}, Vol. III:39--55, 1961.

\bibitem[MS67]{Mardesic:varepsilon-Mappings-and-generalized-manifolds-1967}
S.~Marde\v{s}i\'{c} and J.~Segal.
\newblock {$\varepsilon$-{M}appings and generalized manifolds}.
\newblock {\em Mich. Math. J.}, 14:171--182, 1967.

\bibitem[MSTW14]{MatousekSedgwickTancerWagner:EmbeddabilityS3Decidable-2013}
Ji{\v{r}}{\'{\i}} Matou{\v{s}}ek, Eric Sedgwick, Martin Tancer, and Uli Wagner.
\newblock Embeddability in the $3$-sphere is decidable.
\newblock Preprint, \texttt{arXiv:1402.0815}, 2014.

\bibitem[MTW11]{MatousekTancerWagner:HardnessEmbeddings-2011}
Ji{\v{r}}{\'{\i}} Matou{\v{s}}ek, Martin Tancer, and Uli Wagner.
\newblock {Hardness of embedding simplicial complexes in $\R^d$}.
\newblock {\em J. Eur. Math. Soc.}, 13(2):259--295, 2011.

\bibitem[MW14]{MabillardWagner:TverbergWhitney-2014}
Isaac Mabillard and Uli Wagner.
\newblock Eliminating {T}verberg points, {I}. {A}n analogue of the {W}hitney
  trick.
\newblock In {\em Proc.\ 30th Ann.\ Symp.\ on Computational Geometry}, pages
  171--180, 2014.

\bibitem[MW15]{MabillardWagner15}
Isaac Mabillard and Uli Wagner.
\newblock Eliminating higher-multiplicity intersections, {I}. a {W}hitney trick
  for {T}verberg-type problems.
\newblock Preprint, \href{https://arxiv.org/abs/1508.02349}{\tt
  arXiv:1508.02349}, 2015.

\bibitem[MW16]{MabillardWagner:Elim_II_SoCG-2016}
Isaac Mabillard and Uli Wagner.
\newblock {Eliminating Higher-Multiplicity Intersections, {II}. {T}he Deleted
  Product Criterion in the $r$-Metastable Range}.
\newblock In {\em 32nd International Symposium on Computational Geometry (SoCG
  2016)}, pages 51:1--51:12, 2016.

\bibitem[\"O87]{Ozaydin:Equivariant-maps-for-the-symmetric-group-1987}
Murat \"Ozaydin.
\newblock Equivariant maps for the symmetric group.
\newblock {\em Unpublished manuscript
  \href{http://minds.wisconsin.edu/handle/1793/63829}{\tt
  minds.wisconsin.edu/handle/1793/63829}}, 1987.

\bibitem[PS{\v{S}}07]{Pelsmajer:Removing-even-crossings-2007}
Michael~J. Pelsmajer, Marcus Schaefer, and Daniel {\v{S}}tefankovi{\v{c}}.
\newblock Removing even crossings.
\newblock {\em J. Combin. Theory Ser. B}, 97(4):489--500, 2007.

\bibitem[PS{\v{S}}10]{Pelsmajer:Removing-independently-even-crossings-2010}
Michael~J. Pelsmajer, Marcus Schaefer, and Daniel {\v{S}}tefankovi{\v{c}}.
\newblock Removing independently even crossings.
\newblock {\em SIAM J. Discrete Math.}, 24(2):379--393, 2010.

\bibitem[RCS{\etalchar{+}}97]{ranicki1997hauptvermutung}
AA~Ranicki, A~Casson, D~Sullivan, M~Armstrong, C~Rourke, and G~Cooke.
\newblock The hauptvermutung book.
\newblock {\em Collection of papers by Casson, Sullivan, Armstrong, Cooke,
  Rourke and Ranicki, K-Monographs in Mathematics}, 1, 1997.

\bibitem[RS66]{rourke1966block}
Colin~P Rourke and BJ~Sanderson.
\newblock Block bundles.
\newblock {\em Bulletin of the American Mathematical Society},
  72(6):1036--1039, 1966.

\bibitem[RS68a]{BBI}
C.~P. Rourke and B.~J. Sanderson.
\newblock Block bundles. {I}.
\newblock {\em Ann. of Math. (2)}, 87:1--28, 1968.

\bibitem[RS68b]{BBII}
C.~P. Rourke and B.~J. Sanderson.
\newblock Block bundles. {II}. {T}ransversality.
\newblock {\em Ann. of Math. (2)}, 87:256--278, 1968.

\bibitem[RS68c]{BBIII}
C.~P. Rourke and B.~J. Sanderson.
\newblock Block bundles. {III}. {H}omotopy theory.
\newblock {\em Ann. of Math. (2)}, 87:431--483, 1968.

\bibitem[RS82]{Rourke:Introduction-to-piecewise-linear-topology-1982}
Colin~Patrick Rourke and Brian~Joseph Sanderson.
\newblock {\em Introduction to piecewise-linear topology}.
\newblock Springer Study Edition. Springer-Verlag, Berlin, 1982.
\newblock Reprint.

\bibitem[RS99]{RepovsSkopenkov:NewResultsEmbeddingsManifoldsPolyhedra-1999}
Du{\v{s}}an Repov{\v{s}} and Arkadiy~{B.} Skopenkov.
\newblock New results on embeddings of polyhedra and manifolds into {E}uclidean
  spaces.
\newblock {\em Uspekhi Mat. Nauk}, 54(6(330)):61--108, 1999.

\bibitem[Sar91]{Sarkaria:A-generalized-van-Kampen-Flores-theorem-1991}
K.~S. Sarkaria.
\newblock A generalized van {K}ampen-{F}lores theorem.
\newblock {\em Proc. Amer. Math. Soc.}, 111(2):559--565, 1991.

\bibitem[Sim15]{simon2015average}
Steven Simon.
\newblock Average-{V}alue {T}verberg {P}artitions via {F}inite {F}ourier
  {A}nalysis.
\newblock {\em Preprint \href{http://arxiv.org/abs/1501.04612}{\tt
  arXiv:1501.04612}, to appear in Israel J. Math}, 2015.

\bibitem[Sko98]{Skopenkov_deleted_product}
A.~Skopenkov.
\newblock On the deleted product criterion for embeddability in {${\bf R}^m$}.
\newblock {\em Proc. Amer. Math. Soc.}, 126(8):2467--2476, 1998.

\bibitem[Sko05]{skopenkov2005new}
A~Skopenkov.
\newblock A new invariant and parametric connected sum of embeddings.
\newblock {\em arXiv preprint math/0509621, to appear in Fundamenta
  Mathematicae}, 2005.

\bibitem[Sko08]{Skopenkov:EmbeddingKnottingManifoldsEuclideanSpaces-2008}
Arkadiy~B. Skopenkov.
\newblock Embedding and knotting of manifolds in {E}uclidean spaces.
\newblock In {\em Surveys in contemporary mathematics}, volume 347 of {\em
  London Math. Soc. Lecture Note Ser.}, pages 248--342. Cambridge Univ. Press,
  Cambridge, 2008.

\bibitem[Sko16]{skopenkov2016UserGuide}
A~Skopenkov.
\newblock A user's guide to disproof of topological {T}verberg conjecture.
\newblock {\em Preprint \href{http://arxiv.org/abs/1605.05141}{\tt
  arXiv:1605.05141}}, 2016.

\bibitem[SS92]{Segal:Quasi-embeddings-and-embeddings-of-polyhedra-in-mathbb-Rsp-m-1992}
J.~Segal and S.~Spie{\.z}.
\newblock {Quasi embeddings and embeddings of polyhedra in ${\mathbb R}\sp m$}.
\newblock {\em Topology Appl.}, 45(3):275--282, 1992.

\bibitem[SSS98]{Segal:Embeddings-of-polyhedra-in-mathbb-Rm-and-the-deleted-product-obstruction-1998}
J.~Segal, A.~Skopenkov, and S.~Spie{\.z}.
\newblock {Embeddings of polyhedra in ${\mathbb R}^m$ and the deleted product
  obstruction}.
\newblock {\em Topology Appl.}, 85(1-3):335--344, 1998.

\bibitem[Tut70]{Tutte:TowardATheoryOfCrossingNumbers-1970}
William~T. Tutte.
\newblock Toward a theory of crossing numbers.
\newblock {\em J. Combin. Theory}, 8:45--53, 1970.

\bibitem[Tve66]{Tverberg:A-generalization-of-Radons-theorem-1966}
H.~Tverberg.
\newblock A generalization of {R}adon's theorem.
\newblock {\em J. London Math. Soc.}, 41:123--128, 1966.

\bibitem[Vol96]{Volovikov:On-the-van-Kampen-Flores-theorem-1996}
A.~{Yu.} Volovikov.
\newblock On the van {K}ampen-{F}lores theorem.
\newblock {\em Mat. Zametki}, 59(5):663--670, 797, 1996.

\bibitem[Web67]{Weber:Plongements-de-polyhedres-dans-le-domaine-metastable-1967}
Claude Weber.
\newblock Plongements de polyh{\`e}dres dans le domaine m{\'e}tastable.
\newblock {\em Comment. Math. Helv.}, 42:1--27, 1967.

\bibitem[Zee66]{Zeeman:Seminar-on-combinatorial-topology-1966}
Erik~Christopher Zeeman.
\newblock {\em Seminar on combinatorial topology}.
\newblock Institut des {H}autes {\'E}tudes {S}cientifiques, 1966.

\bibitem[{\v{Z}}iv96]{Zivaljevic:UserGuideEquivariantTopologyCombinatorics-96}
Rade~T.\ {\v{Z}}ivaljevi\'{c}.
\newblock {User's guide to equivariant methods in combinatorics.}
\newblock {\em Publ. Inst. Math. Beograd}, 59(73):114--130, 1996.

\bibitem[{\v{Z}}iv98]{Zivaljevic:UserGuideEquivariantTopologyCombinatorics2-98}
Rade~T. {\v{Z}}ivaljevi\'{c}.
\newblock User's guide to equivariant methods in combinatorics. {I}{I}.
\newblock {\em Publ. Inst. Math. (Beograd) (N.S.)}, 64(78):107--132, 1998.

\bibitem[{\v{Z}}V92]{Zivaljevic:The-colored-Tverbergs-problem-and-complexes-of-injective-functions-1992}
Rade~T.\ {\v{Z}}ivaljevi\'{c} and Sinisa~T. Vre\'{c}ica.
\newblock The colored {T}verberg's problem and complexes of injective
  functions.
\newblock {\em J. Combin. Theory Ser. A}, 61(2):309--318, 1992.

\end{thebibliography}
